\documentclass[reqno, 11pt, a4paper]{amsart}
\usepackage{amsmath, amssymb, amsthm}  
\usepackage{graphicx, color}
\usepackage[foot]{amsaddr}
\usepackage{mathtools}
\usepackage{mathrsfs, enumerate}
\usepackage{hyperref}
\usepackage[truedimen,top=3truecm, bottom=3truecm, left=2.5truecm, right=2.5truecm, includefoot]{geometry}

\usepackage[normalem]{ulem}

\newcommand{\Erase}[1]{\textcolor{blue}{\sout{\textcolor{black}{#1}}}}

\begin{document}

\title[Coupled KPZ from interacting diffusions driven by a single-site potential]{Derivation of coupled KPZ equations from interacting diffusions driven by a single-site potential}
\author[K. Hayashi]{Kohei Hayashi}
\address{RIKEN Interdisciplinary Theoretical and Mathematical Science, 2-1 Hirosawa, Wako, Saitama 351-0198 Japan.} 
\email{kohei.hayashi.vh@riken.jp}
\keywords{KPZ equation, stochastic Burgers equation, interacting diffusion, O'Connell-Yor polymer}
\subjclass[2000]{60H15, 60K35, 82B44}
\maketitle

\theoremstyle{plain}
\newtheorem{theorem}{Theorem}[section] 
\newtheorem{lemma}[theorem]{Lemma}
\newtheorem{corollary}[theorem]{Corollary}
\newtheorem{proposition}[theorem]{Proposition}

\theoremstyle{definition}
\newtheorem{definition}[theorem]{Definition}
\newtheorem{remark}[theorem]{Remark}
\newtheorem{assumption}[theorem]{Assumption}
\newtheorem{example}[theorem]{Example}

\makeatletter
\renewcommand{\theequation}{%
\thesection.\arabic{equation}}
\@addtoreset{equation}{section}
\makeatother

\makeatletter
\renewcommand{\p@enumi}{A}
\makeatother

\newcounter{num}
\newcommand{\Rnum}[1]{\setcounter{num}{#1}\Roman{num}}

\begin{abstract} 
The Kardar-Parisi-Zhang (KPZ) equation is a stochastic partial differential equation which is derived from various microscopic models, and to establish a robust way to derive the KPZ equation is a fundamental problem both in mathematics and in physics. 
As a microscopic model, we consider multi-species interacting diffusion processes, whose dynamics is driven by a nonlinear potential which satisfies some regularity conditions. 
In particular, we study asymptotic behavior of fluctuation fields associated with the processes in the high temperature regime under equilibrium. 
As a main result, we show that when the characteristic speed of each species is the same, the family of the fluctuation fields seen in moving frame with this speed converges to the coupled KPZ equations.
Our approach is based on a Taylor expansion argument which extracts the harmonic potential as a main part.
This argument works without assuming a specific form of the potential and thereby the coupled KPZ equations are derived in a robust way. 
\end{abstract}

\section{Introduction}
Kardar-Pasiri-Zhang~(KPZ) equation is a stochastic partial differential equation~(SPDE) of known scalar-valued function $h =h (t,x)$, where $(t,x)\in [0,\infty) \times \mathbb{R}$, of the following form. 
\begin{equation}
\label{eq:KPZ_intro}
\partial_t h = \nu \partial_x^2 h + \lambda (\partial_x h)^2 + \sqrt{D} \dot{W}(t,x).
\end{equation}
Here $\nu, D >0$ and $\lambda \in \mathbb{R}$ are constants and $ \dot{W}(t,x)$ is the space-time white noise. 
Throughout this paper we consider the one-dimensional setting for the space.  
The KPZ equation \eqref{eq:KPZ_intro} is introduced in \cite{kardar1986dynamic} and it has been gathering interests from both mathematical and physical point of view. 
For a perspective of SPDE theory, KPZ equation~\eqref{eq:KPZ_intro} is thought-provoking since it exhibits a singular behavior. 
More specifically, it turns out that the solution $h$ is not differentiable in $x$ and $\partial_x h$ should be understood in the sense of distribution.
As a result, the singular term $(\partial_xh)^2$ should be defined properly with the help of some renormalization procedure, which formally means subtraction of a divergent factor, and this has been established in several ways~\cite{hairer2013solving, hairer2014theory, gubinelli2015paracontrolled, gubinelli2017kpz, gonccalves2014nonlinear, gubinelli2018energy}. 
A significant feature of the KPZ equation is its universality and the equation is expected to be derived from various interface growth models.  
(See \cite{corwin2012kardar} for a pedagogical review on this perspective.) 
After the celebrated result by \cite{bertini1997stochastic}, which studies density fluctuation fields of weakly asymmetric simple exclusion processes~(WASEP), the derivation of the KPZ equation has been established for several types of microscopic systems~\cite{gonccalves2014nonlinear, gonccalves2015stochastic, diehl2017kardar, jara2019scaling, alberts2014intermediate, corwin2018operatorname, jara2020stationary, ahmed2022microscopic, hayashi2023derivation, gonccalves2023derivation}.  

As a generalization of the equation \eqref{eq:KPZ_intro} we can consider the following system of vector-valued equations of $h = (h^1, \ldots, h^d) \in \mathbb{R}^d$.  
\begin{equation}
\label{eq:coupledKPZ_intro}
\partial_t h^i = \frac{1}{2} \partial_x h^i 
+ \sum_{k, \ell=1}^d \Gamma_{k \ell}^i (\partial_x h^k)(\partial_x h^\ell) 
+ \sqrt{D^i} \dot{W}^i(t, x) .
\end{equation}
Here $(\Gamma^i_{k \ell})_{i, k ,\ell = 1,\ldots, d}$ is a constant coupling tensor and $\{ \dot{W}^i(t,x ) \}_{i =1, \ldots, d}$ is a family of independent space-time white noise. 
The coupled KPZ equation\Erase{s}~\eqref{eq:coupledKPZ_intro}\Erase{,} was introduced in \cite{ertacs1992dynamic} as a random interface model, and then \cite{funaki2017coupled} proved well-posedness of~\eqref{eq:coupledKPZ_intro} on the one-dimensional torus. 
Moreover, in \cite{funaki2017coupled} it is shown that under the so-called \textit{trilinear condition}
\begin{equation}
\label{eq:trilinear_intro}
\Gamma^i_{k \ell} 
= \Gamma^i_{\ell k} 
= \Gamma^k_{i \ell},
\end{equation}
the product of independent space white noise is an invariant measure for the dynamics. 
Under the trilinear condition~\eqref{eq:trilinear_intro}, \cite{gubinelli2020infinitesimal} developed the uniqueness of the stationary energy solution of coupled KPZ equation on the torus, though the uniqueness on the whole line has not proved yet.  
The coupled KPZ equation\Erase{s}~\eqref{eq:coupledKPZ_intro}, similarly to the scalar-valued case~\eqref{eq:KPZ_intro}, is considered to have a universality and it is significant to derive the system~\eqref{eq:coupledKPZ_intro} from microscopic models in mathematically rigorous ways. 
The rigorous derivation of the coupled KPZ equations, however, is not established enough, except for a few models~\cite{bernardin2021derivation, butelmann2022scaling}. 
In these results~\cite{bernardin2021derivation, butelmann2022scaling}, the coupled KPZ equations have been derived under the weak asymmetric regime.
The anti-symmetric part of the infinitesimal generator of these models is accelerated by a small factor, compared to the symmetric part, from which the nonlinear term in the KPZ equation is derived. 
On the other hand, the symmetric part is close to a discrete Laplacian, which macroscopically gives rise to the normal heat diffusion. 
This structure, namely that the generator of the dynamics is decomposed into the sum of the heat diffusion part and the weak asymmetric part, is crucial to obtain the coupled KPZ equations in the limit.   

In this paper, we study a microscopic system under a different regime from these previous models.  
More specifically, we consider $d$-species interacting diffusion processes driven by a nonlinear potential $V:\mathbb{R}^d\to 
\mathbb{R}$, which is in a strong asymmetric regime in the sense that the anti-symmetric part of the generator has the same order as the symmetric part. 
In particular, unlike the previous models~\cite{bernardin2021derivation, butelmann2022scaling} in a weak asymmetric regime, the structure where both the diffusion term and the nonlinear transport term remain in the limiting equation seems lacking for this model. 
However, we can derive the coupled KPZ equations from our model, which is in a strong asymmetric regime, by introducing a proper scaling for the potential and then extracting the heat diffusion as a main part in the high temperature limit.
A novelty of this approach is that it is applicable for possibly a wide class of nonlinear potentials.
Derivation of the scalar-valued KPZ equation in the high temperature was conducted in~\cite{jara2020stationary} for the O'Connell-Yor polymer model, which can be viewed as a scalar-valued interacting diffusion of our model driven by the so-called Toda lattice potential. 
For scalar-valued case, our results generalized this to more general nonlinear potentials.
(See also \cite{corwin2018operatorname, hayashi2023derivation} for models of interacting particles for which the scalar-valued KPZ equation is derived in the strong asymmetric regime.)
Furthermore, we extended the result to multi-species cases and obtained the coupled KPZ equations in the limit.
As far as we know, this is the first result which derives the coupled KPZ equations in a strong asymmetric regime. 

To clarify our idea, let $\beta>0$ be the inverse temperature and we consider the limit $\beta\to0$. 
Moreover, let $V_\beta (x) = \beta^{-2} V(\beta x)$ be a rescaled potential. 
When the driving potential $V$ satisfies $V(0)=\partial_i V (0)=0$ for any $i$, $\partial_{ij}^2V(0)=\delta_{i,j}$, and some regularity condition, we have the following Taylor expansion for the rescaled potential $V_\beta$.
\begin{equation*}
V_\beta(u)
= \frac{1}{2} \|u\|^2
+ \frac{\beta}{3!} \sum_{i_1,i_2,i_3=1,\ldots,d} \partial^3_{i_1i_2i_3}V(0) u^{i_1} u^{i_2} u^{i_3} 
+ O(\beta^2),
\end{equation*}
as $\beta\to 0$. 
The first term in the right-hand side of the last display is the harmonic potential, which typically brings the normal diffusion in the limit.
The order of the second term is controlled by tuning the inverse temperature properly depending on the scaling parameter and this will give rise to the quadratic terms in the limiting equation. 

For our multivariate model we can define fluctuation fields naively from the interacting diffusion processes. 
(See \eqref{eq:diffusion_fluctuation_ori} for the definition.) This family of fluctuation fields, however, diverges due to some interaction between different species. 
More specifically, such difficulty arises when macroscopic velocity of each diffusion process varies. Accordingly, we need an assumption which enables us to take a common speed of the moving frame in order to obtain proper scaling limits, similarly to the previous results for multi-species models~\cite{bernardin2021derivation, butelmann2022scaling}, in particular the corresponding condition is called ‘Frame condition (FC)’ in \cite{bernardin2021derivation}.
Under this assumption we show that in Theorem \ref{thm:diffusion_sbe_derivation_main} the family of the fluctuation fields converges to the coupled KPZ equations. 
Our approach is significant from the view point of its robustness. 
Recall that through the above Taylor expansion argument we could extract the heat diffusion as a main part and it did not require any explicit form of the nonlinear potential except for the conditions on lower order derivatives. 
This supports the universality of the KPZ equation and the coupled KPZ equation, and we conjecture that the equations are derived from wide class of other microscopic models with this robust argument, especially in the high temperature regime.

\subsection*{Organization of the Paper}
First, Section \ref{sec:model} gives a precise description of our model and statement of the main result. 
We recall the notion of stationary energy solution of coupled KPZ, or coupled Burgers equations.
Moreover, we give an outline of the proof of the main theorem based on a martingale decomposition which is obtained from Dynkin's theorem. 
Then, in Section \ref{sec:BG}, we will show the so-called second-order Boltzmann-Gibbs principle as a dynamical estimate, which plays an essential role to derive singular terms in limiting equations. 
After that, we show in Section \ref{sec:diffusion_tightness} that the sequences in the martingale decomposition are tight and then, in Section \ref{sec:identification}, we identify limit points of the sequences.

\subsection*{Notations}
In the sequel, we omit summation symbols if there exists the same letter in both upper and lower indices (Einstein's convention): $a^i_j x^j = \sum_{j=1}^d a^i_j x^j$, for instance.
Moreover, when we give some estimate, we write $C$ as a universal constant, allowing it to change from line to line.

\section{Model and Result}
\label{sec:model}
\subsection{Coupled KPZ Equation}
First, we review the notion of stationary energy solution $h(t,x)=(h^i(t,x))_{1\le i \le d}$ of the following $(1+1)$-dimensional coupled KPZ equation~\cite{ertacs1992dynamic} on the unit torus $\mathbb{T} = \mathbb{R}/\mathbb{Z} \cong [0,1)$. 
\begin{equation}
\label{eq:coupledKPZ}
\partial_t h^i
= \frac{1}{2} \partial_x^2 h^i
+ \frac{1}{2} \Gamma^i_{k \ell} \partial_x h^k \partial_x h^\ell + \dot{W}^i 
, \quad (t,x) \in [0,\infty) \times \mathbb{T}. 
\end{equation}
Here, $\dot{W}(t,x)=(\dot{W}^i(t,x))_{i=1,\ldots,d}$ is the $\mathbb{R}^d$-valued space-time white-noise and we omit the summation symbol over $k$ and $\ell$ by Einstein's convention. 
As a similar object, one can find that the tilt $u^i = \partial_x h^i$ formally satisfies the following coupled stochastic Burgers equations (SBE).
\begin{equation}
\label{eq:coupledSBE}
\partial_t u^i
= \frac{1}{2} \partial_x^2 u^i 
+ \frac{1}{2} \Gamma^i_{k \ell} \partial_x (u^k u^\ell) + \partial_x \dot{W}^i 
, \quad (t,x) \in [0,\infty) \times \mathbb{T}. 
\end{equation}
In the sequel, we assume that the coupling tensor $\Gamma^i_{k \ell}$ is symmetric in all indices: 
\begin{equation}
\label{eq:trilinear}
\Gamma^i_{k \ell}
= \Gamma^i_{\ell k}
= \Gamma^k_{i \ell} .
\end{equation}
The condition \eqref{eq:trilinear} is called the \textit{trilinear condition}.

\begin{definition}
We say that a process $u = \{ u_t = (u^i_t)_{1\le i \le d} :  t \in [0,T]\}$ taking values in $C([0,T], \mathcal{S}^\prime(\mathbb{T})^d)$ satisfies condition \textbf{(S)} if for each $t $ the $\mathcal{S}^\prime(\mathbb{T})$-valued random variables $\{ u^i_t \}_{1 \le i \le d}$ form a family of independent white-noise with variance $1$.
\end{definition}

For a process $u = \{ u_t : t \in [0,T] \}$ satisfying condition \textbf{(S)}, $\varphi \in \mathcal{S}(\mathbb{T})$ and $\varepsilon > 0$, we define 
\begin{equation*}
\mathcal{A}_{s,t}^{\varepsilon,(i_1,i_2)}(\varphi) 
= \int_s^t \int_{\mathbb{T}} u_r^{i_1}(\iota_\varepsilon(x;\cdot)) 
u^{i_2}_r (\iota_\varepsilon(x;\cdot))
\partial_x \varphi (x) dx dr ,
\end{equation*}
where we set $\iota_\varepsilon (x;\cdot) = \varepsilon^{-1} \mathbf{1}_{[x,x+\varepsilon )}(\cdot)$. 
Here, note that the function $\iota_\varepsilon(x;\cdot)$ does not belong to $\mathcal{S}(\mathbb{R})$, though, the above quantity in the last display is well-defined since $u^i_t(\cdot)$ is the space white-noise for each $t\ge 0$ by condition \textbf{(S)}. 

\begin{definition}
Let $u = \{ u_t : t \in [0,T]\}$ be a process satisfying condition \textbf{(S)}. We say that the process $u$ satisfies the energy estimate \textbf{(EC)} if there exists a constant $\kappa > 0$ such that for any $\varphi \in \mathcal{S} (\mathbb{T})$, any $0 \le s \le t \le T$, any $0 < \delta < \varepsilon < 1$ and any $i_1,i_2 \in \{ 1, \ldots, d \}$, 
\begin{equation*}
\mathbb{E}\big[ \big| 
\mathcal{A}_{s,t}^{\varepsilon,(i_1,i_2)}(\varphi) 
- \mathcal{A}_{s,t}^{\delta,(i_1,i_2)}(\varphi)
\big| \big]
\le \kappa (t-s) \varepsilon \| \partial_x \varphi \|^2_{L^2(\mathbb{T}) }.
\end{equation*}
Here $\mathbb{E}$ denotes the expectation with respect to the measure $\mathbb{P}$ of a probability space where the process $u$ lives. 
\end{definition}

The conditions \textbf{(S)} and \textbf{(EC)} are crucial to define the quadratic term in \eqref{eq:coupledSBE} as follows. 

\begin{proposition}
\label{prop:quadratic}
Assume the conditions \textbf{(S)} and \textbf{(EC)}. 
Then there exists an $\mathcal{S}^\prime(\mathbb{T})$-valued process $\{ \mathcal{A}^{(i_1,i_2)}_t : t\in [0,T] \}$ with continuous trajectories such that
\begin{equation*}
\mathcal{A}^{(i_1,i_2)}_{s,t} (\varphi) = \lim_{\varepsilon \to 0}
\mathcal{A}^{\varepsilon, (i_1,i_2)}_{s,t} (\varphi) , 
\end{equation*}
in $L^2(\mathbb{P})$ for any $s,t \in [0,T]$, any $\varphi \in \mathcal{S}(\mathbb{T})$ and any $i_1, i_2 \in \{ 1,\ldots, d\}$.  
\end{proposition}

The proof of Proposition \ref{prop:quadratic} is identical to the scalar-valued case \cite[Theorem 1]{gonccalves2014nonlinear}. 
Using the limiting process which is constructed in Proposition \ref{prop:quadratic}, stationary energy solution of \eqref{eq:coupledSBE} is defined as follows.

\begin{definition}
\label{def:energy_solution}
We say that a process $u = \{ u_t :t \in [0, T]\}$ is a stationary energy solution of the coupled stochastic Burgers equation \eqref{eq:coupledSBE} if 
\begin{enumerate}
\item 
The process $u$ satisfies the conditions \textbf{(S)} and \textbf{(EC)}. 

\item
For all $\varphi \in \mathcal{S}(\mathbb{R})$ and $i \in \{ 1, \ldots, d \}$, the process 
\begin{equation*}
u^i_t(\varphi) - u^i_0(\varphi) - \frac{1}{2} \int_0^t u^i_s(\partial_x^2 \varphi)ds 
+ \frac{1}{2} \Gamma^i_{k \ell} \mathcal{A}_t^{(k,\ell)}(\varphi),
\end{equation*}
is a martingale with respect to the filtration generated by $u$ with quadratic variation $t \| \partial_x \varphi \|^2_{L^2(\mathbb{T})}$.  

\item
Let $\hat{u} = \{ u_{T-t}: t \in [0,T] \}$ be the time-reversed process of $u$. Then, for all $\varphi \in \mathcal{S}(\mathbb{R})$ and $i \in \{ 1, \ldots, d \}$, the process 
\begin{equation*}
\hat{u}^i_t(\varphi) - \hat{u}^i_0(\varphi) - \frac{1}{2} \int_0^t \hat{u}^i_s(\partial_x^2 \varphi)ds 
+ \frac{1}{2} \Gamma^i_{k \ell} \hat{\mathcal{A}}_t^{(k,\ell)}(\varphi),
\end{equation*}
is a martingale with respect to the filtration generated by $\hat{u}$ with quadratic variation $t \| \partial_x \varphi \|^2_{L^2(\mathbb{T})} $, where we set $\hat{\mathcal{A}}_t= \mathcal{A}_T - \mathcal{A}_{T-t}$.  
\end{enumerate}
\end{definition}
Under the trilinear condition \eqref{eq:trilinear}, the uniqueness of energy solution of the coupled Burgers equation \eqref{eq:coupledSBE} is established in \cite{gubinelli2020infinitesimal}. 
(See also \cite{funaki2017coupled} for another way of definition.)

\subsection{Model Description}
In the sequel, we define our model and state the main result. 
We consider dynamics driven by a single-site potential which satisfies the following assumption. 

\begin{assumption}
\label{ass:diffusion_potential}
Let $V: \mathbb{R}^d \to \mathbb{R}$ be a smooth convex function which satisfies the following conditions.
\begin{itemize}
\item Assume $V(0) = \partial_i V(0)=0 $ and $\partial^2_{ij} V(0) =\delta_{i,j}$ hold where $\delta_{i,j}$ denotes the Kronecker delta and $\partial_i$ denotes the derivative with respect to the $i$-th coordinate. 
Moreover, assume the following algebraic constraint for the third derivatives: for each $i_1,i_2,i_3,i_4\in \{1,\ldots,d\}$,
\begin{equation}
\label{eq:diffusion_algebraic_constraint}
\sum_{k=1}^d \partial^3_{ki_1i_2}V(0)
\partial^3_{ki_3i_4}V(0)
= \sum_{k=1}^d \partial^3_{ki_1i_3}V(0)
\partial^3_{ki_2i_4}V(0). 
\end{equation}
\item 
Assume that there exist constants $C_1>0$ and $C_2 \in \mathbb{R}^d$ such that for any $u\in\mathbb{R}^d$
\begin{equation}
\label{eq:diffusion_assumption_lyapunov}
\Delta V(u) \le C_1(V(u)+1) + C_2 \cdot u.    
\end{equation}
\item 
Assume the derivatives of $V(\cdot)$ up to the fifth order has at most exponential growth, that is, there exist constants $\gamma_V,C=C(\gamma_V)>0$ such that 
\begin{equation}
\label{eq:diffusion_assmption_exponential_growth}
\max_{0\le k \le 5} 
\max_{|\alpha|=k } 
\sup_{u\in\mathbb{R}^d} \big| e^{-\gamma_V|u|} \partial^\alpha V(u) \big| < C.
\end{equation}
Here $\partial^\alpha= \partial_1^{\alpha_1} \cdots \partial_d^{\alpha_d}$ and $|\alpha|=\alpha_1+\cdots+\alpha_d$ for each multi-index $\alpha=(\alpha_1,\ldots,\alpha_d)\in \mathbb{Z}_+^d$.  
In addition, we defined $|u|=\sum_{1\le i\le d} \sum_{1\le j \le n}|u^i_j|$ for each $u=(u^i_j)_{1\le i \le d, j\in\mathbb{T}_n}$. 
\end{itemize}
\end{assumption}

In what follows, we fix a potential $V$ which satisfies Assumption~\ref{ass:diffusion_potential}.
Examples of potential which satisfy Assumption~\ref{ass:diffusion_potential} will be given later in Section \ref{sec:diffusion_examples}.

Let $\beta > 0$ be a parameter which is called the inverse temperature.
For a potential which satisfies Assumption \eqref{ass:diffusion_potential}, we set $V_\beta (x)= \beta^{-2} V (\beta x)$. 
Moreover, let $n> 0$ be a scaling parameter and hereafter we take $\beta = n^{-1/2}$. 
In the sequel, we take the infinity temperature limit, which means we let the inverse temperature go to zero. 
In this case, according to Assumption \ref{ass:diffusion_potential}, the scaled potential $V_\beta$ is decomposed into the sum of a quadratic potential and a remainder term which is asymptotically small, by the Taylor expansion, as we mentioned in Introduction.  
Now, let us define our dynamics driven by the potential. 
We consider a periodic boundary condition and let $\mathbb{T}_n \cong \{ 1, \ldots , n \}$ be a discrete torus whose elements are identified by modulo $n$. 
Moreover, let $\mathscr{X} = \mathbb{R}^{d \times n}$ be a state space and we denote its elements by $u =(u^i_j)_{1\le i \le d, 1 \le j \le n}$ and we simply write $u_j = (u^{1}_j, \ldots, u^{d}_j ) \in \mathbb{R}^d$ for each $j = 1, \ldots, n$. 
In this situation, let $L$ be an operator which satisfies  
\begin{equation}
\label{eq:generator}
L f(u) = \frac{1}{2}\sum_{i=1}^d \sum_{j\in \mathbb{T}_n}
(\partial^i_j - \partial^i_{j-1})^2 f(u) 
+ \sum_{i=1}^d \sum_{j\in \mathbb{T}_n} 
(\partial_i V_\beta (u_{j-1}) - \partial_i V_\beta (u_j)) \partial^i_j f(u),
\end{equation}
for each $f \in C^2(\mathscr{X})$.
Here $\partial^i_j$ denotes the derivative with respect to $u^i_j$. 
Then, let $u^n =(u^{n,i}_j(t):t\ge 0)_{i,j}$ be a diffusion process whose infinitesimal generator is given by $L_n = n^2 L$. 
Here, we give a comment on the construction of the above dynamics. 
To show that there is no explosion for the diffusion, it suffices to show the existence of a Lyapunov function $\Psi: \mathscr{X} \to \mathbb{R}_+\coloneqq [0,\infty)$ which satisfies $L\Psi \le C\Psi$ for some $C>0$. 
(See \cite[Chapter~24]{varadhan1980lectures}.)
This is achieved by 
\begin{equation*}
\Psi (u) = \sum_{j\in \mathbb{T}_n} (V(u_j)+1+ (1/C_1)C_2 \cdot u_j),
\end{equation*}
where $C_1$ and $C_2$ are the constants in \eqref{eq:diffusion_assumption_lyapunov} in Assumption \ref{ass:diffusion_potential}.
Indeed, for each $u\in \mathbb{R}^d$, we have $\Psi(u) \ge 0$ since $\Delta V(u)\ge 0$ by convexity and  
\begin{equation*}
\begin{aligned}
L\Psi(u) 
&= \sum_{i=1}^d \sum_{j\in \mathbb{T}_n} (\partial_i V_\beta(u_{j-1}) - \partial_i V_\beta(u_j))\partial_i V_\beta(u_j) 
+ \sum_{j\in \mathbb{T}_n} \Delta V(u_j) \\
&= - \frac{1}{2}\sum_{i=1}^d \sum_{j\in \mathbb{T}_n} (\partial_i V_\beta(u_{j-1}) - \partial_i V_\beta (u_j))^2 
+ \sum_{j\in \mathbb{T}_n} \Delta V(u_j)
\le C_1 \Psi(u),         
\end{aligned}
\end{equation*}
and thus there is no explosion. 
Moreover, the diffusion process is identified with the following system of It\^{o}-type stochastic differential equations. 
\begin{equation}
\label{eq:SDE_model}
du^{i}_j = (\partial_i V_\beta (u_{j-1}) - \partial_i V_\beta(u_j)) dt 
+ (dB_j^{i}(t) - dB_{j-1}^{i}(t) ) ,
\end{equation}
where $(i,j) \in \{ 1,\ldots, d\} \times \{1, \ldots, n\}$ and $\{ B^i_j \}_{1\le i\le d, 1\le j \le n} $ denotes a family of independent standard Brownian motions. 
Next, we mention the family of invariant measures of the dynamics. 

\subsection{Invariant Measure and Static Estimates}
\label{sec:diffusion_static_estimate}
In what follows, we fix $\lambda = (\lambda^i)_{1 \le i \le d} \in \mathbb{R}^d$, and let $\nu_{\beta,\lambda}$ be a product probability measure on $\mathbb{R}^{d \times n}$ which has the following form.  
\begin{equation}
\label{eq:inv_meas}
\nu_{\beta,\lambda} (du) 
= \prod_{j=1}^n \frac{1}{Z_{\beta, \lambda}}
\exp ( - V_\beta (u_j) + \lambda \cdot u_j ) du_j ,
\end{equation}
where 
\begin{equation*}
Z_{\beta, \lambda} = \int_{\mathbb{R}^d} \exp (-V_\beta(u) + \lambda\cdot u) du ,   
\end{equation*}
is the partition function that makes $\nu_{\beta,\lambda}$ a probability measure, and $\beta>0$ is the inverse temperature whose value is taken in such a way that $Z_{\lambda,\beta}$ is finite so that the measure $\nu_{\beta,\lambda}$ is well-defined. 
The next result shows that the convexity of the nonlinear potential $V$ assures that when $\beta$ is sufficiently small the measure $\nu_{\beta,\lambda}$ is well-defined and a uniform exponential moment bound holds. 

\begin{lemma}
\label{lem:diffusion_uniform_moment_bound}
For any $\gamma>0$, there exists $C_\gamma>0$ and $\beta_c>0$ such that 
\begin{equation*}
\sup_{\beta<\beta_c} E_{\nu_{\beta,\lambda}} \big[e^{\gamma |u|} \big] < C_\gamma.
\end{equation*}
\end{lemma}
\begin{proof}
First, we show that the partition function has a uniform bound. 
Let we choose $u_0\in\mathbb{R}^d$ in such a way that $|u_0^i|\ge4|\lambda^i|$ for each $i=1,\ldots,d$.
Moreover, we assume $\beta>0$ is sufficiently small so that 
\begin{equation*}
\beta M_* d |u_0|^2 \le 
\min_{1\le i\le d} |u_0^i| ,
\end{equation*}
where  
\begin{equation*}
M_*= \max_{1\le i_1,i_2,i_3\le d} \sup_{u\in \mathbb{R}^d, |u|\le |u_0| } 
\big| \partial^3_{i_1i_2i_3} V(u) \big|,  
\end{equation*}
and $|u_0|^2=\sum_{1\le i\le d}(u_0^i)^2$ is the usual Euclidean norm.  
Recall that Taylor expansion for $V_\beta(u_0)$ yields
\begin{equation*}
\partial_i V_\beta(u_0)
= u_0^i + \frac{\beta}{2} \sum_{k_1,k_2=1}^d 
\partial^3_{ik_1k_2}V(\delta u_0) u_0^{k_1}u_0^{k_2}.  
\end{equation*}
for some $\delta=\delta(u_0)\in(0,1)$. 
This has the bound 
\begin{equation*}
\mathrm{sgn}(u_0^i) \partial_i V_\beta(u_0)
\ge \mathrm{sgn}(u_0^i) u_0^i - \frac{1}{2} \beta M_* d  |u_0|^2
\ge \frac{|u_0^i|}{2},
\end{equation*}
where we used the elementary bound $\sum_{1\le i,j\le d}a_ia_j=\big(\sum_{1\le i \le d}a_i\big)^2\le d\big(\sum_{1\le i\le d}a_i^2\big)$ for each $a\in\mathbb{R}^d$.
In the sequel, let $\{ e_i\}_{1\le i \le d}$ be the standard normal basis of $\mathbb{R}^d$. 
Then, we have the bound 
\begin{equation*}
\begin{aligned}
\nabla V_\beta \bigg( \sum_{i=1}^d \mathrm{sgn}(u^i) |u^i_0| e_i \bigg) \cdot u
&= \sum_{i=1}^d u^i \partial_i V_\beta \bigg( \sum_{i=1}^d \mathrm{sgn}(u^i) |u^i_0| e_i \bigg) \\
&\ge \sum_{i=1}^d \frac{|u_0^i|}{2} |u^i|
\ge 2 \sum_{i=1}^d |\lambda^i u^i|,
\end{aligned}
\end{equation*}
for each $u\in\mathbb{R}^d$. 
On the other hand, note that the potential $V$ is non-negative by the convexity.
Thus, we have that 
\begin{equation*}
V_\beta(u) \ge 
V_\beta(v) + \nabla V_\beta(v)\cdot (u-v)
\ge \nabla V_\beta(v) \cdot (u-v), 
\end{equation*}
for any $u,v\in\mathbb{R}^d$. 
As a consequence, for each $u\in\mathbb{R}^d$, taking $v= \sum_{1\le i\le d}\mathrm{sgn}(u^i)u_0^i e_i$ in the above bound, we have 
\begin{equation*}
V_\beta(u)
\ge \sum_{1\le i\le d} 2|\lambda^i u^i| - C,
\end{equation*}
with some constant $C=C(u_0)$. 
Hence, the partition function has the bound   
\begin{equation*}
Z_{\beta,\lambda}
= \int_{\mathbb{R}^d} e^{-V_\beta(u) + \lambda\cdot u}
du 
\le C \int_{\mathbb{R}^d} e^{-2\sum_{1\le i\le d}|\lambda^iu^i| + \lambda \cdot u } du. 
\end{equation*}
In particular, the density function  $\varphi_{\beta,\lambda}(u)\coloneqq e^{-V_\beta(u)+\lambda \cdot u}$ inside the integral of $Z_{\beta,\lambda}$ is bounded by an $L^1(\mathbb{R}^d)$-function, which is independent of $\beta$. 
Noting $\lim_{\beta \to0}\varphi_{\beta,\lambda}(u)= \exp (-|u|^2/2 + \lambda\cdot u)$ for each $u\in\mathbb{R}^d$, by the dominated convergence theorem, we have 
\begin{equation*}
\lim_{\beta\to0} Z_{\beta,\lambda}
= \int_{\mathbb{R}^d} \lim_{\beta\to0}\varphi_{\beta,\lambda}(u)du
= e^{|\lambda|^2/2}. 
\end{equation*}

Now, we estimate the exponential moment
\begin{equation*}
E_{\nu_{\beta,\lambda}}\big[ e^{\gamma |u| }\big]
= \frac{1}{Z_{\beta,\lambda}} \int_{\mathbb{R}^d}
e^{ - V_\beta(u) + \gamma |u| + \lambda \cdot u } du .
\end{equation*}
By a similar argument for the partition function, note that the density $\exp (-V_\beta(u)+ \gamma|u|+\lambda \cdot u) $ is bounded by an $L^1(\mathbb{R}^d)$ function which is independent of $\beta$. 
Hence, combining with the boundedness of the partition function we complete the proof. 
\end{proof}

In what follows, we further assume that the inverse temperature $\beta$ is so small that the assertion of Lemma \ref{lem:diffusion_uniform_moment_bound} holds with $\gamma=2\gamma_V$ where $\gamma_V$ is the constant in Assumption \ref{ass:diffusion_potential}. 
It is straightforward that the measure $\nu_{\beta,\lambda}$ is invariant under the dynamics of \eqref{eq:SDE_model}, see \cite{diehl2017kardar}. 
Moreover, by Lemma \ref{lem:diffusion_uniform_moment_bound} we have $\lim_{\beta\to0}E_{\beta,\lambda}[u]=\lambda$.
Hence, the family of invariant measures $\{ \nu_{\beta, \lambda} \}$ is parametrized by density. 
In the sequel, we fix $\lambda$ as a constant and take $\beta=n^{-1/2}$, and we use the short-hand notation $\nu_n = \nu_{n^{-1/2},\lambda}$. 


\subsection{Main Results}
The primary aim of this chapter is to study limiting behavior of fluctuations of our conserved quantities. 
Before that, we give comments on the corresponding hydrodynamics under the Euler scaling.
Let $\pi^{n,i}_t$ be empirical measure on $\mathbb{T}$ which is defined by 
\begin{equation*}
\pi^{n,i}_t (dx)
= \frac{1}{n} \sum_{j\in \mathbb{T}_n} u^{n,i}_j(n^{-1}t) \delta_{j/n}(dx).
\end{equation*}
Here, $\delta$ denotes the Dirac measure.
Then, it is heuristically expected that the hydrodynamic limit equation is given by a transport equation: the process $(\pi^{n,i}_t)_{t\ge 0}$ converges in probability to $(\rho^i(t,x)dx)_{t\ge 0}$ as $n\to \infty$ where the density $\rho^i = \rho^i(t,x)$ satisfies 
\begin{equation*}
\partial_t \rho^i = - \partial_x \rho^i .
\end{equation*}
Meanwhile, define a modified empirical measure $\tilde{\pi}^{n,i}_t$ by 
\begin{equation}
\label{eq:modified_emp}
\tilde{\pi}^{n,i}_t (dx) 
= \frac{1}{n} \sum_{j\in \mathbb{T}_n} u^{n,i}_j(n^{-1/2}t) \delta_{(j-n^{3/2}t)/n}(dx).
\end{equation}
In order to use the Einstein convention, we use the following short-hand notation for the third and fourth derivatives: 
\begin{equation}
\label{eq:omit_derivatives}
\gamma^{i_1}_{i_2 i_3} = (1/2)\partial^3_{i_1 i_2 i_3}V(0), \quad
\delta^{i_1}_{i_2 i_3 i_4} = (1/6) \partial^4_{i_1 i_2 i_3 i_4}V(0) .
\end{equation}
Note that constants $\gamma^{i_1}_{i_2 i_3}$ and $\delta^{i_1}_{i_2 i_3 i_4}$ are symmetric in all indices since one can change the order of differentiation. 
Then, the hydrodynamic limit equation which corresponds to the modified empirical measure \eqref{eq:modified_emp}, is given by 
\begin{equation}
\label{eq:diffusion_hdl_equation_heuristic}
\partial_t \tilde{\rho}^i = - \partial_x \Lambda^i_{i_1} \tilde{\rho}^{i_1} ,
\end{equation}
where the matrix $\Lambda=(\Lambda^{i_1}_{i_2})_{i_1,i_2}$ is defined by 
\begin{equation}
\label{eq:matrix_lambda}
\Lambda^{i_1}_{i_2}= 2 \gamma^{i_1}_{i_2 i_3} \lambda^{i_3}. 
\end{equation} 
Since the matrix $\Lambda$ is symmetric, it is diagonalized by using an orthogonal matrix $Q$ as $Q^{-1}\Lambda Q = \mathrm{diag}[\eta^1, \ldots, \eta^d]$ where $(\eta^i)_{i=1,\ldots, d}$ are eigenvalues of $\Lambda$. 
Here we roughly explain how the matrix $\Lambda$ and the hydrodynamic equation appear in the limit. 
Let us denote the integration with respect to $\tilde{\pi}^{n,i}_t$ as 
\begin{equation*}
\langle \pi^{n,i}_t, G \rangle
= \frac{1}{n} \sum_{j\in\mathbb{T}_n} 
u^{n,i}_j(n^{-1/2}t) G\big( (j-n^{3/2}t)/n\big), 
\end{equation*}
for each smooth function $G:\mathbb{T}\to \mathbb{R}$. 
Then, applying Dynkin's martingale formula, we have that 
\begin{equation*}
\langle \pi^{n,i}_t, G \rangle
- \langle \pi^{n,i}_0, G \rangle 
= \int_0^t (\partial_s + L_n) \langle \pi^{n,i}_s, G \rangle
ds + M^G_t ,
\end{equation*}
where $M^G_t$ is a martingale. 
The integrand in the first term of the last display is calculated as 
\begin{equation*}
\begin{aligned}
(\partial_t + L_n) \langle \pi^{n,i}_t, G\rangle 
&=\frac{1}{\sqrt{n}} \sum_{j\in\mathbb{T}_n}
\bigg( \partial_i V_\beta(u_j(t)) \big( n(G_{j+1}-G_j)(t)\big)
- u^i_j(t) \partial_x G_j(t) \bigg) \\
&= \frac{1}{\sqrt{n}} \sum_{j\in\mathbb{T}_n}
\big( \partial_i V_\beta(u_j(t)) - u^i_j(t) \big)
\partial_x G_j(t) + o_n(1),
\end{aligned}
\end{equation*}
where we used the short-hand notation $G_j(t)=G((j-n^{3/2}t)/n)$ and replaced the discrete derivative by the continuous one. 
Here and in what follows, we denote centered variables with respect to $\nu_n$ by writing bar over a variable, for example, $\overline{u}^i_j=u^i_j-E_{\nu_n}[u^i_j]$. 
Then, note that the quantity inside the parentheses can be expanded as  
\begin{equation*}
\begin{aligned}
\partial_i V_\beta (u_j) - u^i_j
&= \beta \gamma^i_{k_1k_2} u^{k_1}_j u^{k_2}_j + O(\beta^2)\\
&= \beta \gamma^i_{k_1k_2} \overline{u}^{k_1}_j \overline{u}^{k_2}_j + 
\beta \Lambda^i_k u^k_j + O(\beta^2),
\end{aligned}
\end{equation*}
as $\beta\to 0$, and in the second line we neglected some constant term, which does not appear in the limiting equation. 
Hence, the matrix $\Lambda$ appears when centering variables, and, when we take $\beta=n^{1/2}$, the hydrodynamic equation \eqref{eq:diffusion_hdl_equation_heuristic} is expected to be derived. 

Now, we are in a position to consider fluctuation fields of our model.
A natural choice of fluctuation fields naively corresponds to the process $u=(u^i_t:t\in[0,T])_{i=1,\ldots,d}$.
Define $\mathcal{X}^n_\cdot = (\mathcal{X}^{n,1}_\cdot ,\ldots, \mathcal{X}^{n,d}_\cdot) \in C([0,T]: \mathcal{S}^\prime(\mathbb{T})^d)$ by 
\begin{equation}
\label{eq:diffusion_fluctuation_ori}
\mathcal{X}^{n,i}_t (\varphi;f_n)
= \frac{1}{\sqrt{n}} \sum_{j \in \mathbb{T}_n} 
\overline{u}^{n,i}_j (t) \varphi \big ({\textstyle \frac{j-f_nt}{n} }\big),
\end{equation}
for each $\varphi\in \mathcal{S}(\mathbb{T})$ and $i=1,\ldots,d$.
Here the constant $f_n$ is a moving frame which will be taken in a proper way. 
To state our result, let $\Xi=(\Xi^{i_1}_{i_2})_{1 \le i_1,i_2\le d}$ be a matrix which is defined by
\begin{equation}
\label{eq:xi_matrix}
\begin{aligned}
\Xi^{i_1}_{i_2}
&= 3\delta^{i_1}_{i_2k_2k_3}\lambda^{k_2}\lambda^{k_3} 
+ \sum_{k_2=1}^d \frac{14}{5}\delta^{i_1}_{i_2k_2k_2}  
+ \frac{1}{5} \delta^{i_1}_{i_2i_2i_2}  \\
&\quad- 2 \gamma^{i_1}_{i_2k_1} \gamma^{k_1}_{k_2k_3} \lambda^{k_2}\lambda^{k_3} 
- \sum_{k_2=1}^d \frac{18}{5} \gamma^{i_1}_{i_2k_1} \gamma^{k_1}_{k_2k_2} 
- \frac{2}{5} \gamma^{i_1}_{i_2k_1} \gamma^{k_1}_{i_2i_2} . 
\end{aligned}
\end{equation}
Then, the main result of this chapter is the derivation of the coupled stochastic Burgers equation under the condition that the matrices $\Lambda$ and $\Xi$ is proportional to the identity matrix. 

\begin{theorem}
\label{thm:diffusion_sbe_derivation_main}
Assume that the matrices $\Lambda$ and $\Xi$, which are defined by \eqref{eq:matrix_lambda} and \eqref{eq:xi_matrix}, respectively, are proportional to the identity matrix: there exist constants $\eta,\eta^\prime\in \mathbb{R}$ such that
\begin{equation}
\label{eq:diffusion_framing_condition_eta}
\Lambda^{i_1}_{i_2} = \eta \mathbf{1}_{i_1=i_2}, 
\end{equation}
and 
\begin{equation}
\label{eq:diffusion_framing_condition_xi}
\Xi^{i_1}_{i_2} = \eta^\prime \mathbf{1}_{i_1=i_2},
\end{equation}
for each $i_1,i_2\in\{1,\ldots,d\}$. 
Then the sequences of fluctuation fields $\{ \mathcal{X}^n_t (\cdot;f_n):t\in[0,T]\}_n$ with $f_n=n^2+\eta n^{3/2} + \eta^\prime n$ converges in distribution in $C([0,T],\mathcal{S}^\prime(\mathbb{T})^d)$ to the unique stationary energy solution of the coupled stochastic Burgers equation 
\begin{equation}
\label{eq:diffusion_coupled_kpz_thm}    
\partial_t u^i = \frac{1}{2} \partial_x^2 u^i 
- \gamma^i_{i_1 i_2} \partial_x (u^{i_1} u^{i_2})
+ \partial_x \dot{W}^i .
\end{equation}
Here $\{\dot{W}^i\}_{1\le i\le d}$ denotes the family of independent space-time white-noise. 
\end{theorem}

\begin{remark}
When the second framing condition \eqref{eq:diffusion_framing_condition_xi} does not hold, we can formally derive the limiting SPDE with a drift term 
\begin{equation*}
\partial_t u^i = \frac{1}{2} \partial_x^2 u^i 
- \gamma^i_{i_1 i_2} \partial_x (u^{i_1} u^{i_2})
- \Xi^{i}_{i_1} \partial_x u^{i_1} 
+ \partial_x \dot{W}^i  ,
\end{equation*}
instead of \eqref{eq:diffusion_coupled_kpz_thm}. 
However, the uniqueness of the stationary energy solution is lacking when a drift term exists, so that the full convergence of the discrete sequences is not assured.
This is why we imposed the condition \eqref{eq:diffusion_framing_condition_xi} in addition to \eqref{eq:diffusion_framing_condition_eta}. 
\end{remark}

\begin{remark}
The limiting equation~\eqref{eq:diffusion_coupled_kpz_thm} satisfies the so-called trilinear condition 
\begin{equation*}
\gamma^{i_1}_{i_2i_3} =
\gamma^{i_1}_{i_3i_2} = 
\gamma^{i_2}_{i_1i_3},
\end{equation*}
since the potential is smooth and one can change the order of derivatives. 
Under the trilinear condition, \cite{gubinelli2020infinitesimal} proved that the stationary energy solution of \eqref{eq:diffusion_coupled_kpz_thm} is unique. 
In this chapter, the dynamics is considered with the periodic boundary condition whose limiting space is the compact torus, though, central procedures of the derivation of an SPDE will analogously be established even for the full-line case.
The reason for choosing the periodic boundary is because the uniqueness of the stationary energy solution of the coupled KPZ, or SBE, on the real line is not demonstrated, while the derivation procedure is analogous. 
Once the uniqueness of such an infinite volume martingale problem is established, then we expect that for the full-line case \textit{mutatis mutandis} the coupled KPZ equation on the real line is derived.  
\end{remark}

\subsection{Examples}
\label{sec:diffusion_examples}
In the sequel, we give several examples for our main theorem. 
First, we restrict to the $d=1$ case.
Then, note that the algebraic condition \eqref{eq:diffusion_algebraic_constraint}, and the framing conditions \eqref{eq:diffusion_framing_condition_eta} and \eqref{eq:diffusion_framing_condition_xi} become trivial.

\begin{example}
When $d=1$ and consider the case that the nonlinear function $V$ is given as the so-called FPU-$\alpha$ potential $V(u)=u^2/2+ \alpha u^3 + u^4/4$ with $\alpha\in \mathbb{R}$. 
This potential is convex if and only if $3\alpha^2 \le 1$ and satisfies the conditions \eqref{eq:diffusion_assumption_lyapunov} and \eqref{eq:diffusion_assmption_exponential_growth} in Assumption \ref{ass:diffusion_potential}. 
Moreover, the partition function $Z_{\beta,\lambda}$ is finite if $\beta^{-1}\ge 24|\lambda\alpha|$ and the uniform moment bound Lemma \ref{lem:diffusion_uniform_moment_bound} holds with $\beta_c=24(|\lambda|+\gamma)|\alpha|$. 
\end{example}

\begin{example}[The O'Connell-Yor polymer]
When $d=1$ and the potential is given as the Toda lattice potential $V(u)=e^{-u}-1+u$, the dynamics matches that of the O'Connell-Yor polymer~\cite{o2001brownian}.
In this case, the stochastic Burgers equation has been derived in \cite{jara2020stationary}, so that our result is an extension of this result. 
Note that the Toda lattice potential satisfies conditions \eqref{eq:diffusion_assumption_lyapunov} and \eqref{eq:diffusion_assmption_exponential_growth} in Assumption \ref{ass:diffusion_potential} and that the invariant measure for this choice of potential is well-defined when $\beta^{-1}> \lambda$.
Moreover, we can see that $E_{\nu_{\beta,\lambda}}[e^{\gamma |u|}]$ is bounded by a constant which is independent of $\beta$, provided $\beta^{-1}>\gamma+\lambda$. 
In other words, the assertion in Lemma \ref{lem:diffusion_uniform_moment_bound} holds with $\beta_c=(\gamma+\lambda)^{-1}$. 
\end{example}

Now, we give multi-species examples.

\begin{example}
Assume that the third the fourth derivatives of the potential are diagonalized at the origin: there exist $c_3,c_4 \in \mathbb{R} \setminus \{0\}$ such that $\gamma^{i_1}_{i_2i_3}=c_3$ and $\delta^{i_1}_{i_2i_3i_4}=c_4$ if and only if $i_1=i_2=i_3=i_4$. 
In this case, the algebraic constraint \eqref{eq:diffusion_algebraic_constraint} is automatically satisfied. 
Moreover, assume $\lambda^i=\lambda$ for each $i=1,\ldots,d$. 
Then, the conditions \eqref{eq:diffusion_framing_condition_eta} and \eqref{eq:diffusion_framing_condition_xi} hold with $\eta=2c_3 \lambda$ and $\eta^\prime=3\lambda^2(c_4-c_3^2) + (3c_4-c_3^2)$.
\end{example}

\begin{example}
Let $d=2$. 
By symmetry, we can set $\gamma^1_{11}=a$, $\gamma^1_{12}=\gamma^1_{21}=\gamma^2_{11}=b$, 
$\gamma^1_{22}=\gamma^2_{12}=\gamma^2_{21}=c$ and $\gamma^2_{22}=d$. 
We consider the case that $\delta^{i_1}_{i_2i_3i_4}= 2\sum_{k}\gamma^k_{i_1i_2}\gamma^k_{i_3i_4}$.
This is always possible according to the algebraic constraint \eqref{eq:diffusion_algebraic_constraint}. 
Moreover, we set $\lambda=0$ for simplicity. 
Then the condition \eqref{eq:diffusion_framing_condition_eta} becomes trivial and we can simply write $\Xi^{i_1}_{i_2} 
= 2 \sum_{k_1,k_2}\gamma^{i_1}_{k_1k_2}\gamma^{i_2}_{k_1k_2}$ in this case.
Therefore, \eqref{eq:diffusion_algebraic_constraint} and \eqref{eq:diffusion_framing_condition_xi} are rewritten as 
\begin{equation*}
\begin{aligned}
ac + bd = b^2 + c^2 ,\quad
a^2 + b^2 = c^2+d^2, \quad 
ab + 2bc + cd = 0.
\end{aligned}
\end{equation*}
These identities are satisfied, by homogeneity, when $(a,b,c,d)$ are proportional to 
\begin{equation*}
(p(p^2+3), p^2-1 , p(p^2-1), -3p^2-1), 
\end{equation*}
for any $p\in\mathbb{R}$. 
Moreover, note that we can add some higher order polynomials to the above function in order to assure the conditions \eqref{eq:diffusion_assumption_lyapunov}.
After this procedure, the exponential growth condition \eqref{eq:diffusion_assmption_exponential_growth} still holds.
Then, all conditions of Theorem \ref{thm:diffusion_sbe_derivation_main} are satisfied and the coupled SBE is derived from this case. 
Note that in this example, we can find a choice of $\gamma^{i_1}_{i_2i_3}$ and $\delta^{i_1}_{i_2i_3i_4}$ such that the condition \eqref{eq:diffusion_framing_condition_eta} is valid whereas the condition \eqref{eq:diffusion_framing_condition_xi} is not. 
\end{example}

\subsection{Proof Outline}
In the sequel, we write discrete derivative operators as follows. 
\begin{equation*}
\nabla^n \varphi^n_j = \frac{n}{2}(\varphi^n_{j+1} - \varphi^n_{j-1}), \quad 
\Delta^n \varphi^n_j = n^2 (\varphi^n_{j+1} + \varphi^n_{j-1} - 2 \varphi^n_j ). 
\end{equation*}
Moreover, let $W_j = (W_j^1, \ldots, W^d_j)$ where $W^i_j = \partial_i V_\beta (u_j ) $. 
Note that we can compute its expectation as
\begin{equation*}
E_{ \nu_{\beta, \lambda} } [W^i_j] 
= \frac{1}{Z_{\beta,\lambda}} \int_{\mathbb{R}^d}  
\big( \partial_i V_\beta(u_j) - \lambda^i + \lambda^i \big)
e^{-V_\beta(u_j)+ \lambda\cdot u_j} du_j 
= \lambda^i .
\end{equation*}
Then, let $\overline{W}_j = (\overline{W}^1_j ,\ldots, \overline{W}^d_j)$ be the centered variable which is defined by $\overline{W}^i_j = W^i_j - \lambda^i$. 
Recall that we set $L_n = n^2 L$, where $L$ is the infinitesimal generator of the dynamics, which is represented in terms of $W^i_j$ as follows.
\begin{equation*}
L_n = \frac{n^2}{2} \sum_{i=1}^d \sum_{j \in \mathbb{T}_n} 
(\partial^i_{j} - \partial^i_{j-1})^2 
+ n^2 \sum_{i=1}^d \sum_{j \in \mathbb{T}_n} 
(\overline{W}^i_{j-1}- \overline{W}^i_j)\partial^i_j . 
\end{equation*}
Next, let us denote the $L^2 (\nu_n) $-adjoint operator by writing $*$ as super script. 
Then, by a usual integration-by-parts formula, we have $(\partial^i_j)^*=-\partial^i_j + \overline{W}^i_j$.
In particular, a straightforward calculation yields that $L^*$ is represented as 
\begin{equation*}
L_n^* 
= \frac{n^2}{2} \sum_{i=1}^d \sum_{j \in \mathbb{T}_n} 
(\partial^i_j -\partial^i_{j-1})^2 
- n^2 \sum_{i=1}^d \sum_{j \in \mathbb{T}_n} 
(\overline{W}^i_{j-1}- \overline{W}^i_j)\partial^i_{j-1} . 
\end{equation*}
Furthermore, let $S_n = (L_n + L^*_n)/2$ and $A_n = (L_n-L^*_n)/2$ be symmetric and anti-symmetric part of $L_n$, which are represented as  
\begin{equation*}
S_n = 
\frac{n^2}{2} \sum_{i=1}^d \sum_{j \in \mathbb{T}_n} 
(\partial^i_j -\partial^i_{j-1})^2 
+ \frac{n^2}{2} \sum_{i=1}^d \sum_{j \in \mathbb{T}_n} 
\overline{W}^i_j (\partial^i_{j+1} + \partial^i_{j-1} - 2 \partial^i_j ) ,
\end{equation*}
and 
\begin{equation*}
A_n = \frac{n^2}{2} \sum_{i=1}^d \sum_{j \in \mathbb{T}_n} 
\overline{W}^i_j (\partial^i_{j+1} - \partial^i_{j-1} ) ,
\end{equation*}
respectively. 

In the sequel, we simply write $\mathcal{X}^n_t(\varphi)=\mathcal{X}^n_t(\varphi;f_n)$ with a generic framing $f_n$.
According to Dynkin's martingale formula,  
\begin{equation}
\label{eq:martingale}
\mathcal{M}^{n, i}_t (\varphi ) 
= \mathcal{X}^{n,i}_t (\varphi ) - \mathcal{X}^{n,i}_0 ( \varphi ) 
- \int_0^t (\partial_s + L_n) \mathcal{X}^{n,i}_s (\varphi ) ds ,    
\end{equation}
is a martingale with quadratic variation 
\begin{equation}
\label{eq:qv}
\begin{aligned}
\langle \mathcal{M}^{n,i}_t(\varphi) \rangle_t 
& = \int_0^t \big(L_n (\mathcal{X}^{n,i}_s (\varphi))^2 - 2 \mathcal{X}^{n,i}_s (\varphi) L_n \mathcal{X}^{n,i}_s (\varphi)) \big) ds  \\
& = n \int_0^t \sum_{ j \in\mathbb{T}_n} T^-_{f_ns}  (\varphi^n_{j} - \varphi^n_{j-1} )^2 ds .
\end{aligned}
\end{equation}
Here we defined the shift operator $T^-$ by $T^-_{v}\varphi_j=\varphi_{j-v}$ for each $v\in\mathbb{R}$. 
Note that the above quantity \eqref{eq:qv} converges to $t \| \partial_x \varphi  \|^2_{L^2(\mathbb{R})}$ as $n\to \infty$, which yields, as we will see, the martingale part $\mathcal{M}^n_t$ converges in distribution to the space-time white-noise. 
Next, we investigate asymptotic behavior of the third term of the utmost right-hand side of \eqref{eq:martingale}.
We begin with the symmetric part 
\begin{equation}
\label{eq:diffusion_symmetric_part}
\mathcal{S}^{n,i}_t (\varphi ) 
= \int_0^t S_n \mathcal{X}^{n,i}_s (\varphi ) ds 
= \frac{1}{2 \sqrt{n} } \int_0^t \sum_{j\in \mathbb{T}_n} \overline{W}^i_j \Delta^n \varphi^n_j (s ) ds .
\end{equation}
To see that the symmetric part \eqref{eq:diffusion_symmetric_part} is a functional of the fluctuation field \eqref{eq:diffusion_fluctuation_ori}, we make use of the Taylor expansion for the nonlinear potential. 
Recall the definition of $\gamma^i_{i_1 i_2}$ and $\delta^i_{i_1 i_2 i_3}$ given in \eqref{eq:omit_derivatives}.
Then, we can expand $W^i_j =\partial_i V_\beta (u_j)$ in terms of $u^i_j$ as follows.  
\begin{equation}
\label{eq:diffusion_taylor_w}
W^i_j
= u^i_j 
+ \frac{1}{\sqrt{n}} \gamma^i_{i_1 i_2} u^{i_1}_j u^{i_2}_j 
+ \frac{1}{n} \delta^i_{i_1 i_2 i_3} u^{i_1}_j u^{i_2}_j u^{i_3}_j
+\varepsilon_j,
\end{equation}
where $\varepsilon_j$ is a reminder term of Taylor's theorem, which satisfies $| \varepsilon_j| \le Cn^{-3/2}e^{2\gamma_V| u_j|}$. 
This bound follows from the exponential bound of $V$, see Assumption \ref{ass:diffusion_potential}. 
With the help of this expansion, we can replace $W_j$ in \eqref{eq:diffusion_symmetric_part} by $u_j$ whose cost is estimated as follows. 
Write the expansion \eqref{eq:diffusion_symmetric_part} as $W_j=u_j+ \tilde{\varepsilon}_j$ where $\tilde{\varepsilon}_j$ is a reminder term which satisfies $|\tilde{\varepsilon}_j| \le C n^{-1/2}e^{2\gamma_V|u_j|}$. 
Then, by centering variables and then using Schwarz's inequality, the cost of the replacement is estimated as 
\begin{equation*}
\mathbb{E}_n \bigg[\sup_{0\le t\le T}\bigg| \frac{1}{2\sqrt{n}} \int_0^t \sum_{j\in\mathbb{T}_n} \tilde{\varepsilon_j}(s) \Delta^n \varphi^n_j(s) ds \bigg|^2 \bigg]
\le C \frac{T^2}{n} \| \partial^2_x \varphi\|^2_{L^2(\mathbb{T})}, 
\end{equation*}
where we used the uniform exponential moment bound Lemma \ref{lem:diffusion_uniform_moment_bound}.
As a result, it is expected that the symmetric part $\mathcal{S}^{n,i}_t(\varphi)$ is asymptotically close to $(1/2) \mathcal{X}^{n,i}_t (\partial_x^2 \varphi)$ as $n \to \infty$, noting that $\sum_{j} \Delta^n \varphi^n_j = 0$.

Next, we investigate the anti-symmetric part
\begin{equation}
\label{eq:antisymmetric}
\begin{aligned}
\int_0^t (\partial_s + A_n ) \mathcal{X}^{n,i}_s (\varphi ) ds 
&= \frac{1}{\sqrt{n}} \int_0^t \sum_{j \in \mathbb{T}_n } 
\big[ n \overline{W}^i_j \nabla^n \varphi^n_j(s) - \frac{f_n}{n} \overline{u}^{n,i}_j(s) \partial_x \varphi^n_j (s ) \big] ds \\
&= \sqrt{n} \int_0^t \sum_{j \in \mathbb{T}_n } 
\big[ \overline{W}^i_j(s) - \frac{f_n}{n^2} \overline{u}^{n,i}_j(s) \big] \nabla^n \varphi^n_j(s) ds +o_n(1),
\end{aligned}
\end{equation}
where recall that we set $\overline{u}^{n,i}_j = u^{n,i}_j - E_{\nu_n}[u^{n,i}_j]$ and note that replacement of continuous derivatives of $\varphi$ by discrete ones is possible. 
Indeed, note that $\nabla^n \varphi^n_j(t)-\partial_x \varphi^n_j(t) = O(n^{-2})$ holds by the mean value theorem.
Thus, by Schwarz's inequality, we have that
\begin{equation}
\label{eq:derivative_replacement}
\begin{aligned}
& \mathbb{E}_n \bigg[ \sup_{0\le t \le T} \bigg| 
\sqrt{n} \int_0^t \sum_{j \in \mathbb{T}_n } 
\overline{u}^{n,i}_j(s) (\nabla^n \varphi^n_j(s) - \partial_x \varphi^n_j (s))ds \bigg|^2\bigg] \\
& \quad \le C\frac{T}{n^3} \int_0^T \sum_{j\in \mathbb{T}_n} \mathbb{E}_n [\overline{u}^{n,i}_j(t)^2] 
\le C \frac{T^2}{n^2} ,
\end{aligned}
\end{equation}
for some $C>0$. 
To proceed, again we use the Taylor expansion \eqref{eq:diffusion_taylor_w}. 
In fact, we can show that the difference $W^i_j -u^i_j$ is expanded in terms of $W^i_j$ as follows. 
\begin{equation}
\label{eq:w-u_sketch}
\begin{aligned}
W^i_j - u^i_j 
= \frac{1}{\sqrt{n}} \gamma^i_{i_1 i_2} \overline{W}^{i_1}_{j-1} \overline{W}^{i_2}_j 
+ \frac{1}{\sqrt{n}} \Lambda^i_k u^{k}_j
+ \frac{1}{n} \Xi^i_k u^k_j 
+ R_j ,
\end{aligned}
\end{equation}
where recall that the matrices $\Lambda$ and $\Xi$ are given by \eqref{eq:matrix_lambda} and \eqref{eq:xi_matrix}, respectively, and $R_j$ denotes a factor which does not affect the limit. 
The proof of the expansion \eqref{eq:w-u_sketch} will be given in Section \ref{sec:diffusion_taylor_expansion}. 
We define fluctuation field $\mathcal{B}^n_\cdot = (\mathcal{B}^{n,1}_\cdot , \ldots, \mathcal{B}^{n,d}_\cdot )$ associated with the first term of \eqref{eq:w-u_sketch} by
\begin{equation*}
\mathcal{B}^{n,i}_t (\varphi)
= \int_0^t \sum_{j \in \mathbb{T}_n} \gamma^i_{i_1 i_2} 
\overline{W}^{i_1}_{j-1}(s) \overline{W}^{i_2}_j(s) \nabla^n \varphi^n_j (s) ds ,
\end{equation*}
for each $\varphi \in \mathcal{S}(\mathbb{T})$, which brings us the non-linear term of SBE in a similar way as \cite{jara2020stationary}. 
Now, we take $f_n=n^2$ to make use of the Taylor expansion \eqref{eq:w-u_sketch}. 
Then, we obtain the decomposition
\begin{equation}
\label{eq:diffusion_decomposition_original}
\begin{aligned}
\mathcal{X}^{n,i}_t(\varphi)
& = \mathcal{X}^{n,i}_0(\varphi)
+ \mathcal{S}^{n,i}_t(\varphi)
+ \mathcal{B}^{n,i}_t(\varphi)
+ \sqrt{n} \int_0^t \Lambda^i_k \mathcal{X}^{n,k}_s(\partial_x \varphi) ds 
+ \int_0^t \Xi^i_k \mathcal{X}^{n,k}_s(\partial_x \varphi) ds \\
& \quad + \mathcal{M}^{n,i}_t(\varphi)
+ \mathcal{R}^{n}_t(\varphi) ,
\end{aligned}
\end{equation}
where $\mathcal{R}^n_\cdot$ denotes a reminder term which does not affect the limit. 
In what follows, we use the same notation for such a reminder term with abuse of notation. 
Looking the decomposition \eqref{eq:diffusion_decomposition_original}, however, the fourth term in the utmost right-hand side diverges as $n \to \infty$. 
It is impossible to cancel this singular term for any choice of the moving frame $f_n$, except for the case the matrix $\Lambda$ is diagonal.
On the other hand, assume now that the conditions \eqref{eq:diffusion_framing_condition_eta} and \eqref{eq:diffusion_framing_condition_xi} are valid.
Then, taking the moving frame $f_n$ as in Theorem \ref{thm:diffusion_sbe_derivation_main}, the linear terms are canceled and the above martingale decomposition becomes 
\begin{equation}
\label{eq:diffusion_decomposition_diagonal_case}
\begin{aligned}
\mathcal{X}^{n,i}_t(\varphi)
& = \mathcal{X}^{n,i}_0(\varphi)
+ \mathcal{S}^{n,i}_t(\varphi)
+ \mathcal{B}^{n,i}_t(\varphi)
+ \mathcal{M}^{n,i}_t(\varphi)
+ \mathcal{R}^{n}_t(\varphi) .
\end{aligned}
\end{equation}
In forthcoming sections, we will show tightness of each term of the martingale decomposition \eqref{eq:diffusion_decomposition_diagonal_case} and then identify the limit points.

\section{The Second-order Boltzmann-Gibbs Principle}
\label{sec:BG}
We give some preliminaries in this section.  
In particular, we show the second-order Boltzmann-Gibbs principle, which is a key ingredient to show the main theorem. 
Let $\mathcal{C} = C^2(\mathbb{R}^{d \times n})$. 
It is easy to verify that the Dirichlet form, associated with our Markov process, is given by   
\begin{equation}
\label{eq:H1norm}
\| f \|^2_{1,n} = \langle f, - L_n f \rangle_{L^2(\nu_n)}
= \frac{n^2}{2} \sum_{i=1}^d \sum_{j \in \mathbb{T}_n} 
E_{\nu_n} \big[ \big( (\partial^i_{j+1} - \partial^i_{j}) f \big)^2 \big] . 
\end{equation}
(See \cite{jara2020stationary} or \cite{diehl2017kardar} for a detailed computation.)
Moreover, we define the following norm as a dual of $\| \cdot \|_{1,n}$.  
\begin{equation}
\label{eq:H-1norm}
\|f \|^2_{-1,n} = \sup_{f \in \mathcal{C}} 
\big\{ 2 \langle f,g \rangle_{L^2(\nu_n)} - \| g \|^2_{1,n} \big\} .  
\end{equation}
Then, recall that the following Kipnis-Varadhan inequality holds for any mean-zero function $F:[0,T] \to L^2(\nu_n)$. 
\begin{equation*}
\mathbb{E}_n \bigg[ \sup_{0 \le t \le T} 
\bigg| \int_0^T F(s, u^n(s)) ds \bigg|^2 \bigg]
\le C \int_0^T \| F(s, \cdot) \|_{-1,n } ds . 
\end{equation*}
In addition, the following integration-by-parts formula is available. 

\begin{lemma}[Integration by parts]
\label{lem:ibp}
For each $F \in \mathcal{C} $, we have that 
\begin{equation*}
E_{\nu_n} [ \overline{W}^i_j F ] = E_{\nu_n} [\partial^i_j F ]  .
\end{equation*}
\end{lemma}
\begin{proof}
By a usual integration-by-parts formula, we have that
\begin{equation*}
\begin{aligned}
E_{\nu_n} [\partial_j^i F ] 
& = \frac{1}{Z_{\beta,\lambda}} \int_{ \mathbb{R}^{d\times n}}
\partial^i_j F(u) \exp \bigg( - \sum_{j=1}^n (V_{\beta} (u_j) - \lambda \cdot u_j) \bigg) du \\
& = \frac{1}{Z_{\beta,\lambda}} \int_{ \mathbb{R}^{d\times n}}
F(u) (\partial_i V_\beta (u_j) - \lambda^i) \exp \bigg( - \sum_{j=1}^n (V_{\beta} (u_j) - \lambda \cdot u_j) \bigg) du 
= E_{\nu_n} [\overline{W}^i_j F ] .
\end{aligned}
\end{equation*}
\end{proof}
Similarly to the proof of Lemma \ref{lem:ibp}, we have for any $F,G\in \mathcal{C}$ that 
\begin{equation*}
E_{\nu_n} [ G \partial^i_j F ] = - E_{\nu_n} [ F \partial^i_j G ] + E_{\nu_n} [\overline{W}^i_j F G ]  .
\end{equation*}
This immediately implies $(\partial^i_j)^* = - \partial^i_j + \overline{W}^i_j $. 
Let $\overrightarrow{(W^i)}^\ell_j 
= \ell^{-1} \sum_{k=0}^{\ell-1} \overline{W}^i_{j+k}$ be a local average of $\overline{W}^i_j$.
Then, we have the following estimate.

\begin{theorem}[The second-order Boltzmann-Gibbs principle]
\label{thm:2BG}
For each $i_1,i_2 \in \{ 1, \ldots, d \}$, we have that
\begin{equation*}
\begin{aligned}
& \mathbb{E}_n \bigg[ \sup_{0 \le t \le T} \bigg| 
\int_0^t \sum_{j \in \mathbb{T}_n} 
\big( \overline{W}^{i_1}_{j-1} \overline{W}^{i_2}_{j} 
- \overrightarrow{(W^{i_1})}^\ell_j 
\overrightarrow{(W^{i_2})}^\ell_j \big)(s)
\varphi^n_j (s) ds \bigg|^2 \bigg] \\
&\quad \le C \bigg( \frac{\ell}{n^2} + \frac{T}{\ell^2} \bigg) \int_0^T \sum_{j \in \mathbb{T}_n } 
\varphi^n_j (t)^2 dt .
\end{aligned}
\end{equation*}
\end{theorem}

In particular, when the condition \eqref{eq:diffusion_framing_condition_eta} is imposed so that the moving framing can commonly taken for each component, this estimate enables us to replace the fluctuation field $\mathcal{B}^{n,i}_\cdot$ by local averages. 

\begin{proof}[Proof of Theorem \ref{thm:2BG}]
We hereafter omit the dependence on $n$ for simplicity. 
When $i_1=i_2$ holds, the theorem is already shown in \cite[Proposition 1]{jara2020stationary} in the O'Connell-Yor polymer case, and the extension to the general potential is straightforward. 
Hence, it is sufficient to show the estimate for $i_1 , i_2 \in \{ 1, \ldots, d \}$ such that $i_1\neq i_2$. 

Note that the following decomposition is available. 
\begin{equation}
\label{eq:oneblock_decomposition}
\begin{aligned}
 \overline{W}^{i_1}_{j-1} \overline{W}^{i_2}_j
- \overrightarrow{(W^{i_1})}^\ell_j \overrightarrow{(W^{i_2})}^\ell_j 
= \overline{W}^{i_1}_{j-1} 
[\overline{W}^{i_2}_j - \overrightarrow{(W^{i_2})}^\ell_j]
+ (\overrightarrow{W^{i_2}})^\ell_{j} 
[\overline{W}^{i_1}_{j-1} - \overrightarrow{(W^{i_1})}^\ell_j] . 
\end{aligned}
\end{equation}
We find it sufficient to show the one-block estimate, which evaluate the cost to replace $\overline{W}^i_j$ by its local average $(\overrightarrow{W}^i_j)^\ell$.  
Indeed, the first term in \eqref{eq:oneblock_decomposition} is rewritten as 
\begin{equation*}
\begin{aligned}
\sum_{j\in \mathbb{T}_n}
\overline{W}^{i_1}_{j-1} (\overline{W}^{i_2}_j - (\overrightarrow{W^{i_2}})^\ell_j) \varphi_j 
&=\sum_{j \in \mathbb{T}_n} \sum_{k=0}^{\ell-2}
\overline{W}^{i_1}_{j-1} (W^{i_2}_{j+k} - W^{i_2}_{j+k+1})\psi_{k} \varphi_j \\
&= \sum_{j\in \mathbb{T}_n } F^{i_1}_j (W^{i_2}_{j} - W^{i_2}_{j+1}),
\end{aligned}
\end{equation*}
where $\psi_k = (\ell - k - 1)/\ell$ and $F^i_j = \sum_{k=0}^{\ell-2} \varphi_{j-k} \overline{W}^{i}_{j-1-k} \psi_j$. 
Note that the functional $F^{i_1}_j$ is invariant under the operation $\partial^{i_2}_j - \partial^{i_2}_{j+1}$ when $i_1 \neq i_2$. 
Hence, for each $f \in C^2(\mathbb{R}^{d\times n} ) $, applying the integration-by-parts and Schwarz's inequality, we have
\begin{equation*}
\begin{aligned}
&\bigg\langle 2 \sum_{j \in \mathbb{T}_n} \overline{W}^{i_1}_{j-1} 
[\overline{W}^{i_2}_j - \overrightarrow{(W^{i_2})}^\ell_j], f \bigg\rangle_{L^2(\nu_n)} \\
&\quad= 2 E_{\nu_n} \bigg[ \sum_{j \in \mathbb{T}_n } F^{i_1}_j (W^{i_2}_j -W^{i_2}_{j+1}) f \bigg] 
 = 2 E_{\nu_n} \bigg[ \sum_{j \in \mathbb{T}_n } F^{i_1}_j (\partial^{i_2}_j -\partial^{i_2}_{j+1}) f \bigg] \\
&\quad \le \frac{2}{n^2} \sum_{j \in \mathbb{T}_n} E_{\nu_n} \big[ (F_j^{i_1})^2 \big]
+ \frac{n^2}{2} \sum_{j \in \mathbb{T}_n} 
E_{\nu_n} \big[ \big( (\partial^{i_2}_{j+1} - \partial^{i_2}_j ) f \big)^2 \big] .
\end{aligned}
\end{equation*}
Since the second term in the last line is bounded by $\| f\|^2_{1,n}$, the Kipnis-Varadhan inequality yields
\begin{equation*}
\begin{aligned}
& \mathbb{E}_n \bigg[ \sup_{0 \le t \le T} \bigg| \int_0^t \sum_{j \in \mathbb{T}_n} 
\overline{W}^{i_1}_{j-1}(s) 
[\overline{W}^{i_2}_j(s) - \overrightarrow{(W^{i_2})}^\ell_j(s)] 
\varphi_j (s) ds \bigg|^2 \bigg] \\
& \quad \le C \frac{1}{n^2} 
\int_0^T \sum_{i=1}^d \sum_{j \in \mathbb{T}_n} 
E_{\nu_n} [ F^{i}_j(t)^2 ] dt 
\le C \frac{\ell}{n^2} \int_0^T \sum_{j \in \mathbb{T}_n}  \varphi_j (t)^2 . 
\end{aligned}
\end{equation*}
The second term in the utmost right-hand side of \eqref{eq:oneblock_decomposition} can be treated similarly, and thus we complete the proof. 
\end{proof}

\section{Tightness}
\label{sec:diffusion_tightness}
Recall the martingale decomposition \eqref{eq:diffusion_decomposition_diagonal_case}. 
In this section, we demonstrate that he sequences of fluctuation fields $\{\mathcal{M}^{n,i}_\cdot\}_n$, $\{\mathcal{S}^{n,i}_\cdot\}_n$ and $\{\mathcal{B}^{n,i}_\cdot\}_n$ are tight. 
The tightness of the sequence $\{\mathcal{X}^{n,i}_\cdot\}_n$ follows from the tightness of these sequences.

\subsection{Convergence at Fixed Times}
Before showing tightness of the processes in the decomposition \eqref{eq:diffusion_decomposition_diagonal_case}, we consider convergence of fluctuation fields at fixed times. 
Recalling the static estimate in Section \ref{sec:diffusion_static_estimate}, $\nu_n$ converges weakly to the normal distribution with mean $\lambda$ and unit variance, and $\lim_{n \to \infty} \mathrm{Var}_{\nu_n}[u^{n,i}_j] = 1$ for each $i=1,\ldots,d$.
Hence, we conclude that the law of $\mathcal{X}^{n,i}_t(\varphi)$ for each $t$ converges to the space white-noise with unit variance. 
Now the convergence of the fluctuation fields at fixed times follows in a similar way as \cite[Section 4.1.1]{diehl2017kardar}.

\subsection{Martingale Part}
Next, we are in a position to deal with the martingale term $\mathcal{M}^{n,i}_\cdot$. 
Note that according to Mitoma's criterion \cite{mitoma1983tightness}, the sequence $\{ \mathcal{M}^{n,i}_\cdot \}_n$ is tight in $C([0,T], \mathcal{S}^\prime (\mathbb{T}))$ if and only if the sequence $\{ \mathcal{M}^{n,i}_\cdot (\varphi) \}_n$ is tight in $C([0,T], \mathbb{R})$ for all $\varphi \in \mathcal{S}(\mathbb{T})$. 
Recall that $\tilde{\mathcal{M}}^{n,i}_t$ has quadratic variation 
\begin{equation*}
\begin{aligned}
\langle \mathcal{M}^{n,i} (\varphi)\rangle_t
= n \int_0^t \sum_{i=1}^d \sum_{j \in \mathbb{T}_n}
\big( \varphi^{n}_{j}(s) - \varphi^{n}_{j-1}(s) \big)^2
\le C t \| \partial_x \varphi \|^2_{L^2(\mathbb{T})}.
\end{aligned}
\end{equation*}
By the Burkholder-Davis-Gundy inequality, we have that
\begin{equation*}
\mathbb{E}_n \big[ | \mathcal{M}^{n,i}_{t_2} - \mathcal{M}^{n,i}_{t_1} |^p \big]
\le C |t_2- t_1|^{p/2} \| \partial_x \varphi \|^{p}_{L^2(\mathbb{T})} ,
\end{equation*}
for all $p \ge 0$. 
Hence, we complete the proof of the tightness of the martingale part. 


\subsection{Symmetric Part}
Recall that 
\begin{equation*}
\mathcal{S}^{n,i}_t (\varphi ) 
= \frac{1}{2\sqrt{n}} \int_0^t \sum_{j\in \mathbb{T}_n} 
\overline{W}^{i}_j \Delta^n \varphi^{n}_j(s) ds .
\end{equation*}
Then, by a direct $L^2$-computation, we have a bound
\begin{equation*}
\begin{aligned}
\mathbb{E}_n \big[| \mathcal{S}^{n,i}_{t_2}
- \mathcal{S}^{n,i}_{t_1} |^2 \big]
&\le C |t_2 - t_1| n^{-1} \int_{t_1}^{t_2} \sum_{j \in \mathbb{T}_n} \mathbb{E}_n [\overline{W}^{i}_j (s)^2] (\Delta^n \varphi^n_j(s))^2 ds \\
& \le C |t_2-t_1|^2 \| \partial_x \varphi \|^2_{L^2(\mathbb{T})} ,
\end{aligned}
\end{equation*}
from which the tightness follows immediately from the Kolmogorov-Chentsov criterion.

\subsection{Anti-symmetric Part}
Now we show the tightness of the anti-symmetric part. 
By Theorem \ref{thm:2BG} and stationarity, we have that 
\begin{equation*}
\begin{aligned}
& \mathbb{E}_n \bigg[ \bigg| 
 \mathcal{B}^{n,i}_{t_2} (\varphi) 
- \mathcal{B}^{n,i}_{t_1} (\varphi) 
- \int_{t_1}^{t_2} \sum_{j \in \mathbb{T}_n} 
\gamma^{i}_{k_1k_2} 
\overrightarrow{(W^{k_1})}^\ell_j(s)
\overrightarrow{(W^{k_2})}^\ell_j(s)
\nabla^n \varphi^{n}_j (s) ds \bigg|^2 \bigg] \\
&\quad\le C \bigg( \frac{(t_2 - t_1)\ell}{n^2} + \frac{(t_2- t_1)^2}{\ell^2} \bigg) .
\end{aligned}
\end{equation*}
Here note that we omit the summation symbol in $k_1,k_2$ by Einstein's convention. 
On the other hand, after splitting the summation in $j$ into blocks with length $\ell$, an $L^2$ computation yields 
\begin{equation*}
\mathbb{E}_n \bigg[ \bigg| \int_{t_1}^{t_2} 
\sum_{j \in \mathbb{T}_n} 
\gamma^{i}_{k_1k_2} 
\overrightarrow{(W^{k_1})}^\ell_j(s)
\overrightarrow{(W^{k_2})}^\ell_j(s)
\nabla^n \varphi^{n}_j (s) ds \bigg|^2 \bigg]
\le \frac{(t_2 -t_1)^2 n}{\ell} .
\end{equation*}
When $1/n^2 \le t_2-t_1 \le 1$, we take an integer $\ell\sim\sqrt{t_2 - t_1}n$ in the above two estimate to obtain 
\begin{equation*}
 \mathbb{E}_n \big[ \big| 
 \mathcal{B}^{n,i}_{t_2}(\varphi) 
 - \mathcal{B}^{n,i}_{t_1} (\varphi) \big|^2 \big] 
\le C (t_2 -t_1)^{3/2}.
\end{equation*}
Meanwhile, when $t_2 - t_1 \le 1/n^2$, we can directly estimate as
\begin{equation*}
 \mathbb{E}_n \big[ \big| 
 \mathcal{B}^{n,i}_{t_2}(\varphi) 
 - \mathcal{B}^{n,i}_{t_1} (\varphi) \big|^2 \big] 
\le C (t_2 - t_1)^2 n 
\le C (t_2 -t_1)^{3/2},
\end{equation*}
from which tightness follows by the Kolmogorov-Chentsov criterion.

\section{Identification of Limit Points}
\label{sec:identification}
Throughout this section, we assume the framing conditions \eqref{eq:diffusion_framing_condition_eta} and \eqref{eq:diffusion_framing_condition_xi}. 
Recall the martingale decomposition \eqref{eq:diffusion_decomposition_diagonal_case} under the framing conditions.
By tightness we proved in the previous section, there exist processes $\mathcal{X}^i$, $\mathcal{S}^i$, $\mathcal{B}^i$, $\mathcal{M}^i$ and a subsequence of $n$, which is still denoted by the same notation, such that 
\begin{equation*}
\begin{aligned}
\lim_{n \to \infty} 
(\mathcal{X}^{n,i}, \mathcal{S}^{n,i}, 
\mathcal{B}^{n,i}, \mathcal{M}^{n,i})_{1\le i\le d}
= (\mathcal{X}^i ,  \mathcal{S}^i,
\mathcal{B}^i, \mathcal{M}^{i})_{1\le i\le d} ,  
\end{aligned}
\end{equation*}
in distribution. 
Hereafter we identify limit of the terms in the decomposition \eqref{eq:diffusion_decomposition_diagonal_case} by the stationary energy solution of the coupled stochastic Burgers equation \eqref{eq:diffusion_coupled_kpz_thm}.
 
First, for the symmetric part, by tightness of $\{ \mathcal{Y}^{n,i}_\cdot \}_n $, we can easily show that
\begin{equation*}
\mathcal{S}^{i}_t (\varphi) 
= \frac{1}{2} \int_0^t \mathcal{X}^i_s(\partial^2_x \varphi) ds . 
\end{equation*}
Moreover, $\lim_{n \to \infty} \langle \mathcal{M}^{n,i} (\varphi) \rangle_t = t \| \partial_x \varphi \|^2_{L^2(\mathbb{T})}$ yields the convergence of the martingale part to the white-noise. 
Hence, we are left to identify limit of $\{ \mathcal{B}^{n,i}_\cdot \}_n$. 
Define a modified fluctuation field 
\begin{equation*}
\tilde{\mathcal{X}}^{n,i}_t (\varphi)
= \frac{1}{\sqrt{n}} \sum_{j \in \mathbb{T}_n } \overline{W}^i_j(t) \varphi^n_j (t) . 
\end{equation*}
Then, recalling $\iota_\varepsilon(x) = \varepsilon^{-1} \mathbf{1}_{[x, x + \varepsilon)}$, we have 
\begin{equation*}
(\overrightarrow{W^{i}})_j^{\varepsilon n}
= \frac{1}{\varepsilon n} \sum_{k \in \mathbb{T}_n} 
\overline{W}^i_k \mathbf{1}_{ \left[\frac{j-f_n t}{n}, \frac{j-f_n t}{n} + \varepsilon \right)} \big( {\textstyle \frac{k-f_n t}{n}} \big)
= \frac{1}{\sqrt{n}} 
\tilde{\mathcal{X}}^{n,i}_t \big(\iota_\varepsilon ({\textstyle \frac{j-f_n t}{n}}) \big) . 
\end{equation*}
Consequently, we have that 
\begin{equation*}
\begin{aligned}
&\sum_{j \in \mathbb{T}_n} \gamma^i_{k_1k_2} 
\overrightarrow{(W^{k_1})}^{\varepsilon n}_j(t)
\overrightarrow{(W^{k_2})}^{\varepsilon n}_j(t)
 \nabla^n \varphi^n_j(t)\\
&\quad= \frac{1}{n} \sum_{j \in \mathbb{T}_n} \gamma^i_{i_1, i_2} 
\tilde{\mathcal{X}}^{n,i_1}_t \big(\iota_\varepsilon ({\textstyle \frac{j-f_n t}{n}}) \big)
\tilde{\mathcal{X}}^{n,i_2}_t \big(\iota_\varepsilon ({\textstyle \frac{j-f_n t}{n}}) \big)
\nabla^n \varphi^n_j(t) ,
\end{aligned}
\end{equation*}
On the other hand, let us consider another quantity 
\begin{equation*}
\begin{aligned}
&\sum_{j \in \mathbb{T}_n} \gamma^i_{k_1k_2} 
\overrightarrow{(u^{k_1})}^{\varepsilon n}_j(t)
\overrightarrow{(u^{k_2})}^{\varepsilon n}_j (t)
\nabla^n \varphi^n_j(t) \\
&\quad= \frac{1}{n} \sum_{j \in \mathbb{T}_n} \gamma^i_{i_1, i_2} 
\mathcal{X}^{n,i_1}_t \big(\iota_\varepsilon ({\textstyle \frac{j-f_n t}{n}}) \big)
\mathcal{X}^{n,i_2}_t \big(\iota_\varepsilon ({\textstyle \frac{j-f_n t}{n}}) \big)
\nabla^n \varphi^n_j(t).
\end{aligned}
\end{equation*}
Then, we have the following estimate for these quadratic fields.

\begin{lemma}
\label{lem:diffusion_quadratic_fields_difference}
For each $\ell\in\mathbb{N}$, we have that 
\begin{equation*}
\begin{aligned}
&\mathbb{E}_n \bigg[\sup_{0\le t\le T} \bigg| \int_0^t \sum_{j\in\mathbb{T}_n} \gamma^i_{k_1k_2}
\bigg( \overrightarrow{(W^{k_1})}^\ell_j(s)
\overrightarrow{(W^{k_2})}^\ell_j (s)
- 
\overrightarrow{(u^{k_1})}^\ell_j(s)
\overrightarrow{(u^{k_2})}^\ell_j (s)
\bigg) \nabla^n \varphi^n_j (s)
ds \bigg|^2 \bigg]\\
&\quad\le C T^2 \bigg( \frac{n^{1/2}}{\ell} + \frac{n}{\ell^2} \bigg) 
\| \partial_x \varphi\|^2_{L^2(\mathbb{T})}. 
\end{aligned}
\end{equation*}
\end{lemma}

Here, recall that $W^i_j$ is expended in terms of $u^i_j$ as a consequence of Taylor expansion, from which the above estimate straightforwardly follows by the Cauchy-Schwarz inequality.
Note here that the sequence $\{\mathcal{X}^{n,i}: 1 \le i \le d \}_n$ is tight and that any limit point satisfies the condition \textbf{(S)}. 
Hence, we obtain the following limit for each component. 
\begin{equation*}
\mathcal{A}^{\varepsilon, i}_{s,t}
= \lim_{n \to \infty} \int_s^t \sum_{j \in \mathbb{T}_n} 
\gamma^i_{k_1k_2} 
\overrightarrow{(u^{k_1})}^{\varepsilon n}_j(r)
\overrightarrow{(u^{k_2})}^{\varepsilon n}_j (r) 
\nabla^n \varphi^n_j (r) dr . 
\end{equation*}
Though $\iota_\varepsilon $ does not belong to $\mathcal{S}(\mathbb{T})$, this is not problematic since we can approximate $\iota_\varepsilon$ by functions which belong to $\mathcal{S}(\mathbb{T})$.
(See \cite[Section 5.3]{gonccalves2014nonlinear} for more details.)

Now, by Theorem \ref{thm:2BG} and Lemma \ref{lem:diffusion_quadratic_fields_difference}, we have that 
\begin{equation*}
\begin{aligned}
&\mathbb{E} \bigg[ 
\bigg| \mathcal{B}^{i}_t (\varphi) - \mathcal{B}^i_s (\varphi) 
- \int_s^t \sum_{j \in \mathbb{T}_n}  \gamma^i_{k_1k_2} 
\overrightarrow{(u^{k_1})}^{\varepsilon n}_j(r)
\overrightarrow{(u^{k_2})}^{\varepsilon n}_j (r)
\nabla^n \varphi^n_j(r) dr \bigg|^2 \bigg]\\ 
&\quad\le C \bigg( \frac{(t-s)\varepsilon n}{n} 
+ \frac{(t-s)^2 n^{1/2}}{\varepsilon n}
+ \frac{(t-s)^2}{(\varepsilon n)^2} \bigg)
\| \partial_x \varphi \|^2_{L^2(\mathbb{T})} . 
\end{aligned}
\end{equation*}
Taking the limit $n \to \infty$ and then $\varepsilon\to0$, we obtain 
\begin{equation}
\label{eq:BA_difference}
\mathbb{E} \big[ \big| \mathcal{B}^i_t(\varphi) - \mathcal{B}^i_s(\varphi) 
- \mathcal{A}^{\varepsilon,i}_{s,t} (\varphi) \big|^2 \big]
\le C (t-s) \varepsilon \| \partial_x \varphi \|^2_{L^2(\mathbb{T})} . 
\end{equation}
Hence, \textbf{(EC)} follows from the triangle inequality so that Proposition \ref{prop:quadratic} assures the existence of the limit 
\begin{equation*}
\mathcal{A}^i_{t} (\varphi)
= \lim_{\varepsilon \to 0} \mathcal{A}^i_{0, t} (\varphi).
\end{equation*}
Here we note that the above convergence does not hold straightforwardly since the function $\iota_\varepsilon(x;\cdot)$ is not a Schwartz function. 
However, we can approximate this function by functions in $\mathcal{S}(\mathbb{T})$ and justify the convergence. 
Therefore, the estimate \eqref{eq:BA_difference} yields $\mathcal{B} = \mathcal{A}$ and thus the first two conditions of Definition \ref{def:energy_solution} are satisfied for the limiting point.
Since we can repeat the same argument for the time-reversed process, the condition (3) of Definition \ref{def:energy_solution} is analogously obtained.
Hence, the proof of Theorem \ref{thm:diffusion_sbe_derivation_main} is completed.

\section{Taylor Expansion Argument}
\label{sec:diffusion_taylor_expansion}
In this section, we expand $u_j^i$ in terms of $W^i_j=\partial_i V_\beta(u_j)$. 
Recall that we set $V_\beta(u) = \beta^{-2}V(\beta u)$ and $\beta = n^{-1/2}$.
By Assumption \ref{ass:diffusion_potential} and Taylor's theorem, we can expand $W^i_j$ as 
\begin{equation}
\label{eq:Taylor_base}
W^i_j = u^i_j 
+ \frac{1}{\sqrt{n}}\gamma^i_{i_1 i_2} u^{i_1}_j u^{i_2}_j 
+ \frac{1}{n} \delta^i_{i_1 i_2 i_3} u^{i_1}_j u^{i_2}_j u^{i_3}_j
+ O (n^{-3/2}),
\end{equation}
for large $n$.
Here recall that we set $\gamma_{i_2i_3}^{i_1}= \partial^3_{i_1i_2i_3}V(0)/2$ and $\delta_{i_2i_3i_4}^{i_1} = \partial^4_{i_1i_2i_3i_4} V(0)/6$ and abbreviated the summation symbol by Einstein's convention. 
In the above expansion, we can show that the residual $O(n^{-3/2})$-term, say $R_j$, is negligible when $n\to\infty$. 
This is justified by the uniform bound on the exponential moment, recall Lemma \ref{lem:diffusion_uniform_moment_bound}.
Indeed, note that $R_j$ has the bound $|R_j| \le Cn^{-3/2}e^{2\gamma_V|u_j|}$ where $\gamma_V>0$ is the constant in Assumption \ref{ass:diffusion_potential}. 
Then, a direct $L^2$-computation shows 
\begin{equation}
\label{eq:error}
\begin{aligned}
 \mathbb{E}_n \bigg[ \sup_{0 \le t \le T} \bigg| \sqrt{n} \int_0^t \sum_{j \in \mathbb{T}_n}
\overline{R}_j \nabla^n \varphi^n_j(s) \bigg|^2 \bigg] 
\le  C \frac{T^2}{n} \| \partial_x \varphi \|^2_{L^2(\mathbb{T})} ,
\end{aligned}
\end{equation}
where $\overline{R}_j=R_j -E_{\nu_n}[R_j]$ denotes a centered variable, and we used Lemma \ref{lem:diffusion_uniform_moment_bound} which is assumed to hold with $\gamma=2\gamma_V$.
Thus, the error term with order $O(n^{-3/2})$ vanishes in the limit $n \to \infty$. 
From now on, we replace order-two terms in the utmost right-hand side of \eqref{eq:Taylor_base} using $W^i_j$. 
To achieve our goal, first we note that 
\begin{equation}
\label{eq:multiTaylor_ab}
\begin{aligned}
& \frac{1}{\sqrt{n}} a^i_{i_1 i_2} W^{i_1}_j u^{i_2}_j 
+ \frac{1}{n} b^i_{i_1 i_2 i_3} W^{i_1}_j u^{i_2}_j u^{i_3}_j  \\
& \quad = \frac{1}{\sqrt{n}} a^i_{i_1 i_2} u^{i_1}_j u^{i_2}_j 
+ \frac{1}{n} (a^i_{k i_1}\gamma^k_{i_2 i_3} + b^i_{i_1 i_2 i_3}) u^{i_1}_j u^{i_2}_j u^{i_3}_j +R^{(1)}_j,
\end{aligned}
\end{equation}
by \eqref{eq:Taylor_base} where $R^{(1)}_j$ denotes an error term which satisfies the bound \eqref{eq:error}.
In the sequel, we write error terms which do not affect the limit by $R^{(k)}_j$ for each $k \in \mathbb{N} $. 
Now choose $a^i_{i_1i_2} $ and $b^i_{i_1i_2i_3}$ in order that the leading terms of the identity \eqref{eq:multiTaylor_ab} and $W^i_j - u^i_j$ coincide: $a^i_{i_1i_2} = \gamma^i_{i_1i_2}$, $b^i_{i_1i_2i_3} = \delta^i_{i_1i_2 i_3} - \gamma^i_{ki_1} \gamma^k_{i_2i_3}$.
As a result, 
\begin{equation}
\label{eq:w-u}
W^i_j -u^i_j 
= \frac{1}{\sqrt{n}} \gamma^i_{i_1 i_2} W^{i_1}_j u^{i_2}_j 
+ \frac{1}{n} (\delta^i_{i_1i_2 i_3} - \gamma^i_{ki_1} \gamma^k_{i_2i_3}) W^{i_1}_j u^{i_2}_j u^{i_3}_j + R^{(1)}_j . 
\end{equation}
Next, we replace $u_j^i$ by $W^i_j$ in the order-two term in the above display. 
For that purpose, we use identities. 
\begin{equation*}
\begin{aligned}
 L(u^{i_1}_ju^{i_2}_j) 
&= (W^{i_1}_{j-1}-W^{i_1}_j ) u^{i_2}_j + (W^{i_2}_{j-1} - W^{i_2}_j)u^{i_1}_j
+ 2 \mathbf{1}_{i_1=i_2}, \\
 L(u^{i_1}_ju^{i_2}_j u^{i_3}_j) 
&= (W^{i_1}_{j-1}-W^{i_1}_j ) u^{i_2}_j u^{i_3}_j
+ (W^{i_2}_{j-1} - W^{i_2}_j)u^{i_1}_j u^{i_3}_j 
+ (W^{i_3}_{j-1} - W^{i_3}_j)u^{i_1}_j u^{i_2}_j \\
&\quad + 2 u_{i_1}\mathbf{1}_{i_2=i_3} 
+ 2 u_{i_2}\mathbf{1}_{i_3=i_1} 
+ 2 u_{i_3}\mathbf{1}_{i_1=i_2} ,
\end{aligned}
\end{equation*}
which hold for each $i_1, i_2 ,i_3 \in \{1,\ldots, d\}$. 
Here, we note that the terms in the range of $L$ do not affect the limit according to the following result. 

\begin{lemma}
\label{lem:generator}
There exists a constant $C>0$ such that 
\begin{equation*}
\mathbb{E}_n \bigg[ \sup_{0 \le t \le T} \bigg| \int_0^t \sum_{j\in \mathbb{T}_n} L(u^{i_1}_j u^{i_2}_j) \nabla^n \varphi^n_j (s) ds \bigg|^2 \bigg]
\le C \frac{T}{n} ,
\end{equation*}
and 
\begin{equation*}
\mathbb{E}_n \bigg[ \sup_{0 \le t \le T} \bigg| \int_0^t \sum_{j\in \mathbb{T}_n} L(u^{i_1}_j u^{i_2}_j u^{i_3}_j) \nabla^n \varphi^n_j (s) ds \bigg|^2 \bigg]
\le C \frac{T}{n} .
\end{equation*}
\end{lemma}
\begin{proof}
Take an arbitrary function $g$ which belongs to $\mathbb{R}^{d \times n}$. 
Then, by the integration-by-parts formula we have that 
\begin{equation*}
\begin{aligned}
\langle L(u^{i_1}_j u^{i_2}_j) , g \rangle_{L^2(\nu_n)} 
& = \langle (W^{i_1}_{j-1}-W^{i_1}_j ) u^{i_2}_j + (W^{i_2}_{j-1} - W^{i_2}_j)u^{i_1}_j
+ 2 \mathbf{1}_{i_1=i_2} , g \rangle_{L^2(\nu_n)} \\
& = \langle u^{i_2}_j(\partial^{i_1}_{j-1}- \partial^{i_1}_j )g 
+ u^{i_1}_j (\partial^{i_2}_{j-1} - \partial^{i_2}_j) \rangle_{L^2(\nu_n)} .
\end{aligned}
\end{equation*}
Then by Young's inequality, for any $\alpha > 0$ we have that
\begin{equation*}
\begin{aligned}
& \bigg\langle 2 \sum_{j\in \mathbb{T}_n} L(u^{i_1}_j u^{i_2}_j) \nabla^n \varphi^n_j(s)ds, g \bigg\rangle_{L^2(\nu_n)} \\
&\quad \le 4 E_{\nu_n}\bigg[ \sum_{i=1}^d \sum_{j\in \mathbb{T}_n}
\bigg\{ \alpha (u^i_j \nabla^n \varphi^n_j)^2 
+ \frac{1}{\alpha} \big( (\partial^i_{j-1} - \partial^i_j)g\big)^2 \bigg\} \bigg] .
\end{aligned}
\end{equation*}
Taking $\alpha = 8/n^2$, the utmost right-hand side of the last display can be bounded by 
\begin{equation*}
\frac{16}{n^2} \sum_{i=1}^d \sum_{j\in \mathbb{T}_n} E_{\nu_n}[(u^i_j)^2](\nabla^n \varphi^n_j)^2 
 + \| f \|^2_{1, n} .
\end{equation*}
Hence, we obtain the desired bound with the help of the Kipnis-Varadhan inequality.
On the other hand, again by the integration-by-parts, we have that
\begin{equation*}
\begin{aligned}
&\langle L(u^{i_1}_j u^{i_2}_j u^{i_3}_j), g \rangle_{L^2(\nu_n)} \\
&\quad = \langle u^{i_2}_ju^{i_3}_j(\partial^{i_1}_{j-1}- \partial^{i_1}_j )g 
+ u^{i_3}_ju^{i_1}_j (\partial^{i_2}_{j-1} - \partial^{i_2}_j)g 
+ u^{i_1}_ju^{i_2}_j (\partial^{i_3}_{j-1} - \partial^{i_3}_j)g
\rangle_{L^2(\nu_n)} .
\end{aligned}
\end{equation*}
Accordingly, we can obtain the bound for $L(u^{i_1}_j u^{i_2}_ju^{i_3}_j)$ in a similar way, which completes the proof.
\end{proof}

Now recall that $\gamma^i_{i_1i_2}$, $\delta^i_{i_1i_2i_3}$ and $\gamma^i_{ki_1} \gamma^k_{i_2i_3}$ are symmetric in $i_1$, $i_2$ and $i_3$.
Moreover, we notice that constant terms are irrelevant since $\sum_{j} \nabla^n \varphi^n_j = 0$. 
Thus, according to Lemma \ref{lem:generator}, we can shift indices in \eqref{eq:w-u} by adding some small or linear terms as follows. 
\begin{equation*}
\begin{aligned}
W^i_j -u^i_j 
& = \frac{1}{\sqrt{n}} \gamma^i_{i_1 i_2} W^{i_1}_{j-1} u^{i_2}_j 
+ \frac{1}{n} (\delta^i_{i_1i_2 i_3} - \gamma^i_{ki_1} \gamma^k_{i_2i_3}) W^{i_1}_{j-1} u^{i_2}_j u^{i_3}_j \\ 
& \quad 
+ \frac{2}{n} \sum_{i_2=1}^d (\delta^i_{i_1 i_2 i_2} - \gamma^i_{ki_1}\gamma^k_{i_2i_2})u^{i_1}_j  
+ R^{(2)}_j . 
\end{aligned}
\end{equation*}
Again by a Taylor expansion, we have that
\begin{equation*}
\frac{1}{\sqrt{n}} W^{i_1}_{j-1}W^{i_2}_j 
= \frac{1}{\sqrt{n}} W^{i_1}_{j-1}u^{i_2}_j 
+ \frac{1}{n}\gamma^{i_2}_{k_1 k_2}W^{i_1}_{j-1}u^{k_1}_j u^{k_2}_j 
+ R^{(3)}_j, 
\end{equation*}
which yields 
\begin{equation*}
\begin{aligned}
W^i_j -u^i_j 
& = \frac{1}{\sqrt{n}} \gamma^i_{i_1 i_2} W^{i_1}_{j-1} W^{i_2}_j 
+ \frac{1}{n} (\delta^i_{i_1i_2 i_3} - 2\gamma^i_{ki_1} \gamma^k_{i_2i_3}) W^{i_1}_{j-1} u^{i_2}_j u^{i_3}_j \\ 
& \quad + \frac{2}{n}
\sum_{i_2=1}^d (\delta^i_{i_1 i_2 i_2} - \gamma^i_{ki_1} \gamma^k_{i_2 i_2})u^{i_1}_j
+ R^{(4)}_j . 
\end{aligned}
\end{equation*}
Next, we are concerned with order-three terms. 
Note here that one can replace $W^{i_1}_{j-1} u^{i_2}_j u^{i_3}_j$ in the above display by $W^{i_1}_{j-1} W^{i_2}_j W^{i_3}_j$ with an error term satisfying the bound \eqref{eq:error} since each $u^i_j$ is replaced by $W^i_j$ with cost $O(n^{-1/2})$.
In addition, hereafter we focus on expanding the order-three term $W^{i_1}_{j-1} W^{i_2}_j W^{i_3}_j$ by spatially scattered ones such as $W^{i_1}_{j-1} W^{i_2}_j W^{i_3}_{j+1}$. 
For that purpose, we use identities
\begin{equation*}
\begin{aligned}
L(u_{j-1}^{i_1}u_{j}^{i_2} u_{j+1}^{i_3}) 
 &= W^{i_1}_{j-2}W^{i_2}_j W^{i_3}_{j+1} 
 + W^{i_1}_{j-1} W^{i_2}_{j-1} W^{i_3}_{j+1} 
 + W^{i_1}_{j-1} W^{i_2}_j W^{i_3}_{j} \\
 & \quad - 3W^{i_1}_{j-1} W^{i_2}_j W^{i_3}_{j+1} 
 - u^{i_3}_{j}\mathbf{1}_{i_1=i_2} - u^{i_1}_{j}\mathbf{1}_{i_2=i_3}
 + R_j^{(5)}, \\
L(u_{j-1}^{i_1} u_{j-1}^{i_2} u_{j+1}^{i_3}) 
 &= W^{i_1}_{j-2}W^{i_2}_{j-1} W^{i_3}_{j+1} 
 + W^{i_1}_{j-1} W^{i_2}_{j-2} W^{i_3}_{j+1} 
 + W^{i_1}_{j-1} W^{i_2}_{j-1} W^{i_3}_{j} \\
 &\quad - 3W^{i_1}_{j-1} W^{i_2}_{j-1} W^{i_3}_{j+1}
 + 2u^{i_3}_j \mathbf{1}_{i_1=i_2}
 + R_j^{(6)}, \\
L(u_{j-1}^{i_1} u_{j-1}^{i_2} u_{j}^{i_3}) 
 &= W^{i_1}_{j-2}W^{i_2}_{j-1} W^{i_3}_{j} 
 + W^{i_1}_{j-1} W^{i_2}_{j-2} W^{i_3}_{j} 
 + W^{i_1}_{j-1} W^{i_2}_{j-1} W^{i_3}_{j-1} \\
 &\quad - 3 W^{i_1}_{j-1}W^{i_2}_{j-1} W^{i_3}_{j}
 + 2u^{i_3}_j (\mathbf{1}_{i_1=i_2} - \mathbf{1}_{i_1=i_2=i_3}) 
 + R_j^{(7)}, \\
L(u_{j}^{i_1} u_{j}^{i_2} u_{j}^{i_3}) 
 &= W^{i_1}_{j-1} W^{i_2}_j W^{i_3}_j 
 + W^{i_1}_j W^{i_2}_{j-1} W^{i_3}_j 
 + W^{i_1}_j W^{i_2}_j W^{i_3}_{j-1}\\
 &\quad - 3 W^{i_1}_{j-1} W^{i_2}_{j-1} W^{i_3}_{j-1} 
 + 2u^{i_1}_j \mathbf{1}_{i_2=i_3}
 + 2u^{i_2}_j \mathbf{1}_{i_3=i_1}
 + 2u^{i_3}_j \mathbf{1}_{i_1=i_2}
 + R_{j}^{(8)} . 
\end{aligned}
\end{equation*}
Here note that we have $u^i_j-W^i_j = O(n^{-1/2})$ and that the cost for shifting $W_j^3$ to $W_{j-1}^3 $ in the fourth identity is negligible. 
As a consequence, the reminder terms are irrelevant to the limit. 
Combining everything together, we have that 
\begin{equation}
\label{eq:order3expansion}
\begin{aligned}
& \frac{28}{3} W^{i_1}_{j-1}W^{i_2}_{j}W^{i_3}_{j}
+ \frac{1}{3} W^{i_1}_{j}W^{i_2}_{j-1}W^{i_3}_{j}
+ \frac{1}{3} W^{i_1}_{j}W^{i_2}_{j}W^{i_3}_{j-1}\\
&\quad = 9L(u^{i_1}_{j-1}u^{i_2}_{j}u^{i_3}_{j+1})
+ 3L(u^{i_1}_{j-1}u^{i_2}_{j-1}u^{i_3}_{j+1}) 
+ L(u^{i_1}_{j-1}u^{i_2}_{j-1}u^{i_3}_{j}) 
+ \frac{1}{3} L(u^{i_1}_{j}u^{i_2}_{j}u^{i_3}_{j}) \\
&\qquad -9W^{i_1}_{j-2}W^{i_2}_{j}W^{i_3}_{j+1}
+ 27 W^{i_1}_{j-1}W^{i_2}_{j}W^{i_3}_{j+1}
- 3 W^{i_1}_{j-2}W^{i_2}_{j-1}W^{i_3}_{j+1} 
- 3 W^{i_1}_{j-1}W^{i_2}_{j-2}W^{i_3}_{j+1} \\
& \qquad - W^{i_1}_{j-2}W^{i_2}_{j-1}W^{i_3}_{j} 
- W^{i_1}_{j-1}W^{i_2}_{j-2}W^{i_3}_{j}
+ 2u^{i_3}_j \mathbf{1}_{i_1=i_2=i_3} \\
&\qquad + \frac{25}{3} u^{i_1}_j \mathbf{1}_{i_2=i_3}
 - \frac{2}{3} u^{i_2}_j \mathbf{1}_{i_3=i_1}
 + \frac{1}{3} u^{i_3}_j \mathbf{1}_{i_1=i_2}
 + R_{j}^{(9)} .
\end{aligned}
\end{equation}
The next result shows that order-three terms that are spatially scattered and centered are negligible.

\begin{lemma}
There exists a constant $C>0$ such that 
\begin{equation*}
\mathbb{E}_n \bigg[ \sup_{0 \le t \le T} \bigg| \int_0^t \sum_{j \in \mathbb{T}_n} 
\overline{W}^{i_1}_{j-1}(s) \overline{W}^{i_2}_j(s) \overline{W}^{i_3}_{j+1}(s)
\nabla^n \varphi^n_j (s) ds \bigg|^2 \bigg]
\le C \bigg( \frac{T \ell}{n^2} + \frac{T^2}{\ell^2} \bigg) \| \partial_x \varphi \|^2_{L^2(\mathbb{T})}. 
\end{equation*}
\end{lemma}
\begin{proof}
For the case $i_1 = i_2 = i_3$, the assertion is already shown in \cite[Appendix]{jara2020stationary}.   
Thus, it is sufficient to demonstrate for the other cases.
When two of $i_1$, $i_2$ and $i_3$ are equal, the assertion is a consequence of Theorem \ref{thm:2BG} and a direct $L^2$-computation for local averages. 
On the other hand, the case $i_1$, $i_2$ and $i_3$ are distinct is essentially the same as the one-block estimate. 
(See the proof of Theorem \ref{thm:2BG} or \cite[Lemma 3]{jara2020stationary}.)
\end{proof}

Note that the other order-three terms in the utmost right-hand side of \eqref{eq:order3expansion} can be treated in the same way. 
Moreover, similarly to Lemma \ref{lem:generator}, we can show that the terms in the range of $L$ do not affect the limit. 
Hence,  
\begin{equation*}
\begin{aligned}
W^i_j -u^i_j 
 = \frac{1}{\sqrt{n}} \gamma^i_{i_1 i_2} W^{i_1}_{j-1} W^{i_2}_j 
+ \frac{1}{n} (\delta^i_{i_1i_2 i_3} - 2\gamma^i_{ki_1} \gamma^k_{i_2i_3}) W^{i_1}_{j-1} W^{i_2}_j W^{i_3}_{j+1} 
+ \frac{1}{n} C^i_{i_1} u^{i_1}_j 
+ R^{(10)}_j  ,
\end{aligned}
\end{equation*}
where we set 
\begin{equation*}
C^i_{i_1} u^{i_1}_j 
= \sum_{i_2=1}^d \bigg( \frac{14}{5} \delta^i_{i_1 i_2 i_2} - \frac{18}{5} \gamma^i_{ki_1} \gamma^k_{i_2 i_2} \bigg) u^{i_1}_j
+ \bigg( \frac{1}{5} \delta^i_{i_1 i_1 i_1} - \frac{2}{5} \gamma^i_{ki_1} \gamma^k_{i_1 i_1} \bigg) u^{i_1}_j . 
\end{equation*}
Recalling $ E_{\nu_n} [W^i_j] = \lambda^i$ we replace variables $W^i_j$'s by centered ones.
Then we have that 
\begin{equation*}
\begin{aligned}
W^i_j -u^i_j 
& = \frac{1}{\sqrt{n}} \gamma^i_{i_1 i_2} \overline{W}^{i_1}_{j-1} \overline{W}^{i_2}_j 
+ \frac{1}{n} (\delta^i_{i_1i_2 i_3} - 2\gamma^i_{ki_1} \gamma^k_{i_2i_3}) \overline{W}^{i_1}_{j-1} \overline{W}^{i_2}_j \overline{W}^{i_3}_{j+1} \\ 
& \quad + \frac{2}{\sqrt{n}} \gamma^i_{i_1i_2} \lambda^{i_2} W^{i_1}_j
+ \frac{3}{n} (\delta^i_{i_1i_2 i_3} - 2\gamma^i_{ki_1} \gamma^k_{i_2i_3})
\lambda^{i_2} \lambda^{i_3} W^{i_1}_j 
+ \frac{1}{n} C^i_{i_1} u^{i_1}_j 
+ R^{(11)}_j . 
\end{aligned}
\end{equation*}
We use the same expansion for the third and fourth terms in the right-hand side of the last identity. 
Then we have 
\begin{equation}
\begin{aligned}
W^i_j -u^i_j 
& = \frac{1}{\sqrt{n}} \gamma^i_{i_1 i_2} \overline{W}^{i_1}_{j-1} \overline{W}^{i_2}_j 
+ \frac{1}{n} (\delta^i_{i_1i_2 i_3} - 2\gamma^i_{ki_1} \gamma^k_{i_2i_3}) \overline{W}^{i_1}_{j-1} \overline{W}^{i_2}_j \overline{W}^{i_3}_{j+1} \\ 
& \quad + \frac{2}{\sqrt{n}} \gamma^i_{i_1i_2} \lambda^{i_2} u^{i_1}_j
+ \frac{1}{n} (3\delta^i_{i_1i_2 i_3} - 2\gamma^i_{ki_1} \gamma^k_{i_2i_3})
\lambda^{i_2} \lambda^{i_3} u^{i_1}_j 
+ \frac{1}{n} C^i_{i_1} u^{i_1}_j 
+ R^{(12)}_j . 
\end{aligned}
\end{equation}
Hence, we obtain an expansion
\begin{equation}
\label{eq:expansion}
\begin{aligned}
W^i_j - u^i_j 
= \frac{1}{\sqrt{n}} \gamma^i_{i_1 i_2} \overline{W}^{i_1}_{j-1} \overline{W}^{i_2}_j 
+ \frac{2}{\sqrt{n}} \gamma^i_{i_1 i_2} \lambda^{i_1} u^{i_2}_j
+ \frac{1}{n} \Xi^i_k u^k_j 
+ R^{(13)}_j ,
\end{aligned}
\end{equation}
where the matrix $\Xi = (\Xi^{i_1}_{i_2})_{i_1,i_2 = 1,\ldots, d}$ is computed as 
\begin{equation*}
\begin{aligned}
\Xi^{i_1}_{i_2}
&= C^{i_1}_{i_2} 
+ 3 (\delta^{i_1}_{i_2k_1k_3} - 2\gamma^{i_1}_{i_2k_1} \gamma^{k_1}_{k_2k_3}) \lambda^{k_2} \lambda^{k_3}\\
&= \sum_{k_2,k_3}  
\bigg( \frac{14}{5} \delta^{i_1}_{i_2 k_2 k_3} - \frac{18}{5} \gamma^{i_1}_{i_2k_1} \gamma^{k_1}_{k_2k_3} \bigg) 
\mathbf{1}_{k_2=k_3}
+ \sum_{k_2,k_3} \bigg( \frac{1}{5} \delta^{i_1}_{i_2 k_2 k_3} 
- \frac{2}{5} \gamma^{i_1}_{i_2k_1} \gamma^{k_1}_{k_2k_3} \bigg) 
\mathbf{1}_{i_2=k_2=k_3} \\
&\quad+ \sum_{k_2,k_3} (3\delta^{i_1}_{i_2k_1k_3} - 2\gamma^{i_1}_{i_2k_1} \gamma^{k_1}_{k_2k_3}) \lambda^{k_2} \lambda^{k_3}\\
&= \sum_{k_2,k_3} \delta^{i_1}_{i_2k_2k_3}
\bigg( 3\lambda^{k_2}\lambda^{k_3} 
+ \frac{14}{5} \mathbf{1}_{k_2=k_3} 
+ \frac{1}{5} \mathbf{1}_{k_2=k_3=i_2} \bigg) \\
&\quad- \sum_{k_2,k_3} \gamma^{i_1}_{i_2k_1} \gamma^{k_1}_{k_2k_3}
\bigg( 2\lambda^{k_2}\lambda^{k_3} 
+ \frac{18}{5} \mathbf{1}_{k_2=k_3} 
+ \frac{2}{5} \mathbf{1}_{k_2=k_3=i_2} \bigg). 
\end{aligned}
\end{equation*}
This is the matrix we defined in \eqref{eq:xi_matrix}. 

\section*{Acknowledgments}
This work was supported by JSPS KAKENHI Grant Number JP22J12607 and the Research Institute for Mathematical Sciences, an International Joint Usage/Research Center located in Kyoto University.
Additionally, the author would like to thank Makiko Sasada and Hayate Suda for giving him fruitful comments and suggestions.

\bibliographystyle{abbrv}
\bibliography{reference} 

\end{document}